\documentclass[12pt,reqno]{amsart}
\usepackage{amssymb,amsthm,amsmath,amsxtra,dsfont,appendix,bbm,stix}
\usepackage{hyperref} 
\usepackage{mathrsfs}
\usepackage{xcolor}

\numberwithin{equation}{section}
\makeatletter

\usepackage{hyperref}

\makeatletter
\def\@tocline#1#2#3#4#5#6#7{\relax
  \ifnum #1>\c@tocdepth 
  \else
    \par \addpenalty\@secpenalty\addvspace{#2}%
    \begingroup \hyphenpenalty\@M
    \@ifempty{#4}{%
      \@tempdima\csname r@tocindent\number#1\endcsname\relax
    }{%
      \@tempdima#4\relax
    }%
    \parindent\z@ \leftskip#3\relax \advance\leftskip\@tempdima\relax
    \rightskip\@pnumwidth plus4em \parfillskip-\@pnumwidth
    #5\leavevmode\hskip-\@tempdima
      \ifcase #1
       \or\or \hskip 1em \or \hskip 2em \else \hskip 3em \fi%
      #6\nobreak\relax
      \dotfill
      \hbox to\@pnumwidth{\@tocpagenum{#7}}
    \par
    \nobreak
    \endgroup
  \fi}
\makeatother

\newtheorem{theorem}{Theorem}[section]
\newtheorem{lemma}[theorem]{Lemma}
\newtheorem{proposition}[theorem]{Proposition}

\theoremstyle{definition}

\theoremstyle{remark}

\newcommand\R{{\ensuremath {\mathbb R} }}

\newcommand\1{{\ensuremath {\mathds 1} }}

\renewcommand\phi{\varphi}

\newcommand{\cC}{\mathcal{C}}

\newcommand{\cE}{\mathcal{E}}

\newcommand{\bx}{\mathbf{x}}

\renewcommand{\epsilon}{\varepsilon}

\newcommand{\norm}[1]{ \left| \! \left| #1 \right| \! \right| }

\renewcommand{\ge}{\geqslant}
\renewcommand{\le}{\leqslant}
\renewcommand{\geq}{\geqslant}
\renewcommand{\leq}{\leqslant}

\newcommand{\eps}{\varepsilon}

\newcommand{\bra}[1]{\langle #1|}
\newcommand{\ket}[1]{|#1\rangle}

\newcommand{\VDW}{V_{\rm DW}}
\newcommand{\cEDW}{\cE_{\rm DW}}
\newcommand{\EDW}{E_{\rm DW}}
\newcommand{\hDW}{h_{\rm DW}}
\newcommand{\muex}{\mu_{\rm ex}}

\numberwithin{equation}{section}

\begin{document}

\title{The Hartree functional in a double well}

\author[A. Olgiati]{Alessandro Olgiati}
\address{Universit\'e Grenoble-Alpes \& CNRS,  LPMMC, F-38000 Grenoble, France}
\email{alessandro.olgiati@lpmmc.cnrs.fr}

\author[N. Rougerie]{Nicolas Rougerie}
\address{Universit\'e Grenoble-Alpes \& CNRS,  LPMMC, F-38000 Grenoble, France}
\email{nicolas.rougerie@lpmmc.cnrs.fr}

\date{December, 2020}

\begin{abstract}
We consider a non-linear Hartree energy for bosonic particles in a symmetric double-well potential. In the limit where the wells are far apart and the potential barrier is high, we prove that the ground state and first excited state are given to leading order by an even, respectively odd, superposition of ground states in single wells. The corresponding energies are separated by a small tunneling term that we evaluate precisely.
\end{abstract}

\maketitle

\tableofcontents

\date{April, 2020}

\maketitle

\tableofcontents

\section{Introduction}

Both as a non-linear analysis problem in its own right, and as a basic input to a companion paper~\cite{OlgRouSpe-20}, we are interested in the low energy states of the bosonic Hartree energy functional
\begin{equation}\label{eq:Hartree func}
\cEDW [u]=\int_{\mathbb{R}^d} |\nabla u(x)|^2\,dx+\int_{\mathbb{R}^d} V_{\mathrm{DW}}(x)|u(x)|^2\,dx+\frac{\lambda}{2}\iint_{\mathbb{R}^d\times\mathbb{R}^d} |u(x)|^2 w(x-y)|u(y)|^2\,dxdy,
\end{equation}
with $\lambda,w\geq 0$ a coupling constant and a repulsive pair interaction potential. The crucial feature we tackle is that we take $\VDW$ to be a double-well potential defined as ($\ell$ and $r$ stand for left and right)
\begin{equation}\label{eq:double pot}
V_{\mathrm{DW}}(x)=\min\left\{V_\ell(x),V_r(x)\right\},
\end{equation}
where for some $s\ge2$
\begin{equation}\label{eq:leftrightpot}
V_\ell(x)=\left|x+\bx\right|^s \quad\text{and} \quad V_r(x)=\left|x-\bx\right|^s.
\end{equation}
Here $\bx \in \R^d$ is of the form 
\begin{equation}\label{eq:bx}
\bx =  \left(\frac{L}{2},0,\ldots,0\right)
\end{equation}
for a large\footnote{Chosen depending on a particle number $N$ in~\cite{OlgRouSpe-20}.} parameter $L\to +\infty$. Hence $\VDW$ models a potential landscape with two wells, both the distance and the energy barrier between them being large, and becoming infinitely so in the limit. 

In~\cite{RouSpe-16,OlgRouSpe-20} we are primarily concerned with the mean-field limit of the many-boson problem in such a double-well potential. As input to the second paper~\cite{OlgRouSpe-20} we use crucially several properties of the ground state problem 
\begin{equation}\label{eq:Hartree ener}
\EDW =\inf\left\{\mathcal{E}_\mathrm{DW}[u]\;|\;u\in H^1(\mathbb{R}^d)\cap L^2\big(\mathbb{R}^d,V_{\mathrm{DW}}(x)\,dx\big),\int_{\mathbb{R}^d} |u(x)|^2\,dx=1\right\}
\end{equation}
and of the associated low energy states. Namely, let $u_+$ be the (unique modulo a constant phase, fixed so as to have $u_+ >0$) minimizer for~\eqref{eq:Hartree ener} and 
\begin{equation}\label{eq:Hartree hamil}
\hDW := - \Delta + \VDW + \lambda w * |u_+|^2 
\end{equation}
the associated mean-field Hamiltonian (functional derivative of $\cEDW$ at $u_+$). One easily shows that $\hDW$ has compact resolvent, and we study its eigenvalues and eigenfunctions. 

The Euler-Lagrange variational equation for $u_+$ reads 
$$
\hDW u_+ = \mu_+ u_+
$$
with 
\begin{equation}\label{eq:muplus}
\mu_+ = \EDW + \frac{\lambda}{2} \iint_{\mathbb{R}^d\times\mathbb{R}^d} |u_+(x)|^2 w(x-y)|u_+(y)|^2\,dxdy.
\end{equation}
Since $\hDW$ has a positive ground state (unique up to phase, see~\cite[Section~XIII.12]{ReeSim4}), and $u_+$ is chosen positive, it follows that $\mu_+$ is the lowest eigenvalue of $\hDW$, with corresponding eigenfunction $u_+$.  

We denote $\mu_-$ the smallest eigenvalue above $\mu_+$, $u_-$ an associated eigenfunction, and $\muex$ the third eigenvalue. We aim at proving 
\begin{itemize}
	\item that $\EDW$, $\mu_+$ and $\mu_-$ are given to leading order in terms of the ground state problem in a single well (left or right).
	\item asymptotics for the first spectral gap:
	\begin{equation}\label{eq:annonce 1}
	\mu_- - \mu_+ \underset{L\to \infty}{\to} 0  
	\end{equation}
	with a precise rate (both as an upper and lower bound).
	\item asymptotics for the associated eigenfunctions: that they both converge to superpositions of eigenfunctions of the single wells and that 
	\begin{equation}\label{eq:annonce 2}
	\norm{|u_+| - |u_-|}\underset{L\to \infty}{\to} 0
	\end{equation}
	in suitable norms, and with a precise optimal rate.
	\item a $L$-independent lower bound to the second spectral gap:
	\begin{equation}\label{eq:annonce 3}
	\muex-\mu_- \geq C, \mbox{ independently of } L.
	\end{equation}
\end{itemize}
The spectral theory of Hamiltonians with multiple wells has a long and rich history, selected references most relevant to the following being~\cite{MorSim-80,Davies-82,ComSei-78,AveSeil-75,Harrell-78,Harrell-80,HelSjo-84,HelSjo-85,Simon-84}. See also~\cite{DimSjo-99,Helffer-88} for reviews. Corresponding non-linear results are also available~\cite{Daumer-91,Daumer-94,Daumer-96}, but we have not found proofs of the aforementioned bounds for the setting just described (dictated by the model of interest in~\cite{RouSpe-16,OlgRouSpe-20}). 

Typically, and in particular regarding results with the level of precision we aim at, the analysis in the aforementioned references is performed in a semi-classical regime, namely one studies the spectral properties of 
\begin{equation}\label{eq:semiclassic}
-\hbar ^2  \Delta + V  
\end{equation}
as $\hbar \to 0$, with $V$ a \emph{fixed} multi-well potential. Say the above, symmetric, $\VDW$ but with $L$ fixed.  One obtains that at leading order the eigenvalues are grouped in pairs around the eigenvalues corresponding to a single well (with appropriate modifications for more than two wells, asymmetric wells, or degenerate one-well eigenvalues). This corresponds to eigenfunctions being strongly suppressed in the \emph{classically forbidden region} far from the wells. The (small) splitting between pairs of eigenvalues can be estimated with some precision, and corresponds to the \emph{tunnel effect}, due to quantum eigenfunctions being small but non-zero in the classically forbidden region. That is, quantum mechanically, there is a flux of particles through potential barriers, that is manifested in a lifting of classical energy degeneracies.  

In fact, if $u_{j,+}$ and $u_{j,-}$ are the eigenfunctions corresponding respectively to the smallest and largest eigenvalue in the $j$-th pair, one has 
\begin{equation}\label{eq:annonce +}
u_{j,+} \simeq \frac{u_{j,\ell} + u_{j,r}}{\sqrt{2}} 
\end{equation}
and 
\begin{equation}\label{eq:annonce -}
u_{j,-} \simeq \frac{u_{j,\ell} - u_{j,r}}{\sqrt{2}} ,
\end{equation}
with $u_{j,\ell}$ and $u_{j,r}$ the $j$-th eigenfunction of (respectively) the left and right well. The results on eigenvalues are a reflection of this fact.

Our main results~\eqref{eq:annonce 1}-\eqref{eq:annonce 2}-\eqref{eq:annonce 3} (stated more precisely below) are adaptations of the above well-known findings to the case at hand, namely $\hbar$ fixed and $L\to \infty$. For the applications in~\cite{OlgRouSpe-20} we need the optimal rates in~\eqref{eq:annonce 1} and~\eqref{eq:annonce 2}, i.e. to exactly identify the order of magnitude of the tunneling term. To a large extent, the sequel is an adaptation of known techniques, but we face two main new difficulties:
\begin{itemize}
 \item the fact that we start from the \emph{non-linear} Hartree problem.
 \item the lack of semi-classical WKB expansions for single-well eigenfunctions, that are essentially fixed in our setting.  
\end{itemize}
The second point is particularly relevant to the derivation of the optimal rates in~\eqref{eq:annonce 1} and~\eqref{eq:annonce 2}. 

\bigskip

\noindent \textbf{Acknowledgments:} We thank Dominique Spehner for useful discussions and the joint work on the companion paper~\cite{OlgRouSpe-20}. A mistake in a previous version was kindly pointed out to us by Jean Cazalis and Mathieu Lewin. Funding from the European Research Council (ERC) under the European Union's Horizon 2020 Research and Innovation Programme (Grant agreement CORFRONMAT No 758620) is gratefully acknowledged.

\section{Main results}\label{sect:main}

We carry on with the previous notation, and also denote 
\begin{equation}\label{eq:simple pot}
V(x)=|x|^s,
\end{equation}
with $s\ge2$, our single-well potential, appropriately translated in~\eqref{eq:leftrightpot}, recalling that  
\begin{equation*}
\bx=\big(L/2,0,\dots,0\big)\in\mathbb{R}^d.
\end{equation*}
As regards interactions, we consider them repulsive, i.e. assume $\lambda \geq 0$ and let  $w\in L^\infty(\mathbb{R}^d)$ with compact support such that ($\widehat w$ stands for the Fourier transform)
\begin{equation}\label{eq:assum w}
w\ge0,\quad \widehat w\ge 0.
\end{equation}
Regularity assumptions could be relaxed to some extent, but we do not pursue this.

We consider the Hartree functional in the double-well~\eqref{eq:Hartree func}
The existence of a minimizer for~\eqref{eq:Hartree ener} follows from standard techniques ~\cite[Theorem 11.8]{LieLos-01}, combined with the fact that $V_\mathrm{DW}$ prevents mass losses at infinity. The uniqueness of the minimizer $u_+$ up to a constant phase factor follows from the assumption $\widehat w\ge0$. Let $u_+$ be the unique minimizer. Being unique, it is even under reflections across the $x_1=0$ hyperplane.
We also know that $u_+$ is the unique ground state of the mean-field double-well Hamiltonian~\eqref{eq:Hartree hamil},  i.e.
\begin{equation*}
h_\mathrm{DW}u_+=\mu_+ u_+.
\end{equation*}
Due to the growth of $V_\mathrm{DW}$ the Hamiltonian $h_\mathrm{DW}$ has compact resolvent. We call $u_-$ and $u_\mathrm{ex}$ the eigenvectors whose corresponding energies $\mu_-$ and $\mu_\mathrm{ex}$ are, respectively, the first and second eigenvalue of $h_\mathrm{DW}$ above $\mu_+$. In other words
\begin{equation} \label{eq:h_expansion}
h_\mathrm{DW}=\mu_+\ket{u_+}\bra{u_+}+\mu_-\ket{u_-}\bra{u_-}+\mu_\mathrm{ex}\ket{u_\mathrm{ex}}\bra{u_\mathrm{ex}}+\sum_{n\ge4}\mu_n\ket{u_n}\bra{u_n},
\end{equation}
where
\begin{equation*}
\mu_+<\mu_-\le \mu_\mathrm{ex} \le \mu_n,\;\forall n,\qquad\text{and}\; \{u_+,u_-,u_\mathrm{ex},u_4,u_5,\dots\}\;\text{form an o.n. basis}.
\end{equation*}
Since $|u_+|^2$ is symmetric under reflections across the $x_1=0$ hyperplane, the Hamiltonian $h_\mathrm{DW}$ commutes with those. We can thus choose each eigenvector $u_+,u_-,u_\mathrm{ex},u_n$, $n\ge4$ to be either symmetric or anti-symmetric with respect to such a reflection. In particular, $u_+$ being positive, it must be symmetric.

We will also consider Hartree functionals with external potential $V_\ell$ or $V_r$, that is,
\begin{equation}\label{eq:left right Hartree}
\begin{split}
\mathcal{E}_\ell[u]=\int_{\mathbb{R}^d} |\nabla u(x)|^2\,dx+\int_{\mathbb{R}^d} V_{\ell}(x)|u(x)|^2\,dx+\frac{\lambda}{2}\iint_{\mathbb{R}^d\times\mathbb{R}^d} w(x-y)|u(x)|^2|u(y)|^2\,dxdy\nonumber\\
\mathcal{E}_r[u]=\int_{\mathbb{R}^d} |\nabla u(x)|^2\,dx+\int_{\mathbb{R}^d} V_{r}(x)|u(x)|^2\,dx+\frac{\lambda}{2}\iint_{\mathbb{R}^d\times\mathbb{R}^d} w(x-y)|u(x)|^2|u(y)|^2\,dxdy.
\end{split}
\end{equation}
We will use combinations of the minimizers of $\mathcal{E}_\ell$ and $\mathcal{E}_r$ to approximate the function $u_+$. To this end, we define minimal energies at mass $1/2$
\begin{equation}\label{eq:Hartree left right}
\begin{split}
E_\ell=\inf\Big\{\mathcal{E}_\mathrm{\ell}[u]\;|\;u\in H^1(\mathbb{R}^d)\cap L^2\big(\mathbb{R}^d,V_{\mathrm{\ell}}(x)\,dx\big),\int_{\mathbb{R}^d} |u(x)|^2\,dx=\frac{1}{2}\Big\}\\
E_r=\inf\Big\{\mathcal{E}_\mathrm{r}[u]\;|\;u\in H^1(\mathbb{R}^d)\cap L^2\big(\mathbb{R}^d,V_{\mathrm{r}}(x)\,dx\big),\int_{\mathbb{R}^d} |u(x)|^2\,dx=\frac{1}{2}\Big\}.
\end{split}
\end{equation}
As for the full double-well problem, our assumptions on $V$ and $w$ are sufficient to deduce the existence and uniqueness of a minimizer using standard methods in the calculus of variations. Since the functionals $\mathcal{E}_\ell$ and $\mathcal{E}_r$ coincide up to a space translation of the external potential,
\begin{equation*}
E_\ell=E_r=\mathcal{E}_\ell[u_\ell]=\mathcal{E}_r[u_r],
\end{equation*}
where $u_\ell$ and $u_r$ are, respectively, the unique positive ground states of
\begin{equation}\label{eq:single well hamil}
h_\ell=-\Delta+V_\ell+\lambda w*|u_\ell|^2,\qquad h_r=-\Delta+V_r+\lambda w*|u_r|^2
\end{equation}
with ground state energies $\mu_\ell=\mu_r$. Again, since the functionals coincide up to a translation, the minimizers coincide up to a translation, i.e.,
\begin{equation*}
u_\ell(x)=u_r(x-2\bx) =u_0( x - \bx),
\end{equation*}
where $u_0$ is the minimizer obtained by setting $\bx=0$ in~\eqref{eq:left right Hartree}.
We are adopting here the normalization $\|u_r\|_{L^2}^2=\|u_\ell\|_{L^2}^2=1/2$, which implies
\begin{equation*}
\langle u_{\ell},h_{\ell}u_{\ell}\rangle=\frac{\mu_{\ell}}{2},\qquad \langle u_{r},h_{r}u_{r}\rangle=\frac{\mu_{r}}{2}.
\end{equation*}

Next we define the main small parameter (in the limit $L\to \infty$) entering our analysis. Associated to~\eqref{eq:simple pot} is a semi-classical Agmon distance~\cite{Agmon-82,DimSjo-99,Helffer-88}
\begin{equation}\label{eq:Agmon}
A(x)=\int_0^{|x|}\sqrt{V(r')}\,dr'=\frac{1}{1+s/2}|x|^{1+s/2}.
\end{equation}
The above governs the decay at infinity of eigenfunctions of the single-well Hamiltonians~\eqref{eq:single well hamil}. Accordingly it sets the $L$-dependence of the tunneling term (splitting between eigenvalue pairs)
\begin{equation}\label{eq:tunnel}
\boxed{T := e^{-2 A \left(\frac{L}{2}\right) }  \underset{L\to \infty}{\to} 0.}
\end{equation}
This is the energetic contribution of classically forbidden regions: $e^{-A \left(\frac{L}{2}\right) }$ is the order of magnitude of double-well wave-functions close to the potential barrier at $x_1 = 0$ (i.e. at distances $L/2$ from the potential wells). It has to be squared for the tunneling term is essentially an overlap of two such wave functions. We will express all our estimates in terms of the above parameter.

Similarly one can associate a distance to the double-well potential~\eqref{eq:double pot}
\begin{equation} \label{eq:A_DW}
A_\mathrm{DW}(x)=\begin{cases}
A(x-\bx)\quad&x_1\ge0\\
A(x+\bx)\quad&x_1\le0.
\end{cases}
\end{equation}
The value $A_{\mathrm{DW}}(x)$ represents the Agmon distance $A$ between the point $x$ and the closest of the two bottoms of the wells, namely, either $\bx$ or $-\bx$. In Section \ref{sect:refined} we will need to introduce a further refinement of $A_\mathrm{DW}$, namely the distance within the potential landscape $V_\mathrm{DW}$ between \emph{any} two points.

%

We shall prove the following result, for space dimensions $1 \leq d \leq 3$ (the upper restriction only enters through the Sobolev embedding, and we certainly believe it could be relaxed).

\begin{theorem}[\textbf{Hartree problem in a double-well}] \label{thm:main}\mbox{}\\
	Assume $1\le d \le 3$. We take $\varepsilon > 0$ to stand for an arbitrarily small number, fixed in the limit $L\to \infty$. Generic constants $c_\eps, C_\eps >0$ only depend on this number. We have
	
	\begin{itemize} 
		\item[$(i)$]\textbf{Bounds on the first spectral gap.}
		\begin{equation}
		c_\varepsilon T^{1+\varepsilon}\le\mu_--\mu_+\le C_\varepsilon T^{1-\varepsilon}. \label{eq:first_gap}
		\end{equation}
		\item[$(ii)$]\textbf{Bounds on the second spectral gap.}
		\begin{equation}
		\mu_\mathrm{ex}-\mu_-\ge C.\label{eq:second_gap}
		\end{equation}
		independently of $L$.
		\item[$(iii)$]\textbf{Convergence of lower eigenvectors.}
		\begin{align} \label{eq:L1_convergence}
		\left\||u_+|^2-|u_-|^2\right\|_{L^1}\le C_\varepsilon T^{1-\varepsilon}\\
		\label{eq:L2_convergence}
		\left\||u_+|-|u_-|\right\|_{L^2}\le C_\varepsilon T^{1/2-\varepsilon}\\ \label{eq:Linfty_convergence}
		\left\||u_+|-|u_-|\right\|_{L^\infty}\le C_\varepsilon T^{1/2-\varepsilon}.
		\end{align}
		
	\end{itemize}
\end{theorem}

A few comments:
\begin{enumerate}
	\item As mentioned above, corresponding results for the semi-classical setting have a long history~\cite{DimSjo-99,Helffer-88}. Obtaining the (almost) sharp lower bound in~\eqref{eq:first_gap} in this case usually relies on WKB expansions, unavailable in the present context. We however need this sharp bound in~\cite{OlgRouSpe-20} and have to come up with an alternative method.
	\item The relevance of the definition~\eqref{eq:tunnel} is vindicated by~\eqref{eq:first_gap}. With extra effort one should be able to show that $T$ gives the order of magnitude of the first spectral gap up to at most logarithmic corrections.
	\item Item (iii) is also crucial in~\cite{OlgRouSpe-20}, in particular~\eqref{eq:L1_convergence}. It reflects the expectation~\eqref{eq:annonce +}-\eqref{eq:annonce -}, i.e. that $u_+$ and $u_-$ mostly differ by a sign change in a half-space. This will be put on a rigorous basis later, following~\cite{HelSjo-84,HelSjo-85}. With a suitable choice of $u_{j\ell}, u_{j,r}$ we indeed vindicate~\eqref{eq:annonce +}-\eqref{eq:annonce -}, with remainders $O(T^{1+\varepsilon})$. Then~\eqref{eq:L1_convergence} follows, using also decay estimates for the product $u_{j,\ell} u_{j,r}$. The less sharp estimates~\eqref{eq:L2_convergence}-\eqref{eq:Linfty_convergence} are mostly stated for illustration (and will serve as steps in the proof).
	\item The results in Theorem \ref{thm:main} do not depend, in their essence, on the particular form $w*|u|^2$ of the non-linearity. A possible generalization, modulo a number of adaptations in the proof, would include a cubic local (Gross-Pitaevskii) non-linearity corresponding to $w(x)=\delta(x)$.
\end{enumerate}

\medskip

The following statement on higher eigenvalues/eigenfunctions follows from  variants of the arguments establishing Theorem \ref{thm:main}, as we quickly explain in Appendix \ref{app:higher}. We denote by $\mu_j ^{\ell}, j = 1 \ldots \infty$ the eigenvalues of the left Hamiltonian $h_\ell$ (identical to those of $h_r$), $m(j)$ their multiplicities  and 
$$ M(k) = \sum_{j=1}^k m(j),$$
with the convention that $M(0) = 0$. 

\begin{theorem}[\textbf{Higher spectrum}]\label{thm:higher}\mbox{}\\
Assume $1\le d\le 3$. Let $k\geq 1$ and $\mu_{2M(k-1)+1}, \ldots, \mu_{2M(k)}$ the ordered eigenvalues of $\hDW$ in a corresponding spectral window (counted with multiplicities).  We have:
	\begin{itemize} 
		\item[$(i)$]\textbf{Bounds on small spectral gaps.} For all $2M(k-1)+1 \leq j \leq 2M(k)$
		\begin{equation}
		\left|\mu_{j} - \mu_{k} ^{\ell} \right| \underset{L\to \infty}{\to} 0. 
		\end{equation}
		\item[$(ii)$]\textbf{Bounds on large spectral gaps.} For all $2M(k-1)+1 \leq j \leq 2M(k)$
		\begin{equation}
		\mu_{2 M(k) +1 } - \mu_\mathrm{j}\ge C_k 
		\end{equation}
		for some constant $C_k>0$ independent of $L$.
\item[$(iii)$]\textbf{Convergence of higher eigenvectors.} One can pick an  eigenbasis $u^+_1, u^-_1,\ldots, u^+_{m(k)}, u^- _{m(k)}$ of $\1_{\mu_{2M(k-1)+1} \leq \hDW \leq \mu_{2M(k)}} \hDW$ such that for all $1\leq m \leq m (k)$
\begin{equation}\label{eq:higher CV}
\left\Vert  u^{\pm}_m  - \frac{u_m^{\ell} \pm u_m^{r}}{\sqrt{2}} \right\Vert_{L^2} \underset{L\to \infty}{\to} 0,
\end{equation}
where $u_1^{\ell}, \ldots, u_{m(k)}^{\ell}$ form an orthonormal basis of $\1_{h^{\ell} = \mu^{\ell}_k}h^{\ell}$ and $u_1^{r}, \ldots, u_{m(k)}^{r}$ are their reflections.	Moreover
		\begin{align}
		\left\||u_{m}^+|^2-|u_{m}^-|^2\right\|_{L^1} \underset{L\to \infty}{\to}&0\\
		\left\||u_{m}^+|-|u_{m}^-|\right\|_{L^2} \underset{L\to \infty}{\to}&0 \\ \label{eq:Linfty_convergence2}
		\left\||u_{m}^+|-|u_{m}^-\right\|_{L^\infty}\underset{L\to \infty}{\to}&0,
		\end{align}
and 
\begin{equation}
		\int_{x_1 \le 0}\left| u_{m}^+ + u_{m} ^- \right|^2\,dx \to 0,  \int_{x_1 \ge 0}\left| u_{m}^+-u_{m}^-\right|^2\,dx\underset{L\to \infty}{\to} 0.
		\end{equation}		
	\end{itemize}
\end{theorem}

\medskip

We do not state convergence rates here, but believe the same rates as in Theorem~\ref{thm:main} can be achieved, for $k$ fixed in the limit $L\to \infty$ (or, better said, for convergence rates whose $k$-dependence is left unspecified). We do not pursue the details, nor the dependence on $k$, for we do not need this in our applications~\cite{OlgRouSpe-20}. Certainly, if the eigenvalues are taken high enough in the spectrum ($k\to \infty$ fast enough as $L\to \infty$) the two-mode approximation~\eqref{eq:annonce +}-\eqref{eq:annonce -}, on which the result relies, breaks down.  

\medskip

The rest of the paper contains the proof of Theorem~\ref{thm:main}, organized as follows:
\begin{itemize}
 \item Section~\ref{sect:bounds}: optimal bounds on the decay of eigenfunctions far from the potential wells, and first consequences thereof.
 \item Section~\ref{sec:gaps}: proof of Items (i) and (ii) in Theorem~\ref{thm:main}. The hardest part is the lower bound on the first gap in Item (i).
 \item Section~\ref{sect:refined}: adaptation of techniques of Helffer-Sj\"ostrand~\cite{HelSjo-85} to deduce Item (iii) from the previous bounds. 
 \item Appendix~\ref{sect:app}: a collection of straightforward consequences of the decay estimates of Section~\ref{sect:bounds}.
\end{itemize}

Finally, in Appendix~\ref{app:higher}, we briefly sketch the additional ingredients needed for the proof of Theorem \ref{thm:higher}.

\section{Preliminary estimates} \label{sect:bounds}

\subsection{Regularity and uniformity results} \label{subsect:regularity}

We start by stating and proving in this subsection a number of important properties of the eigenvectors and eigenvalues of $h_\mathrm{MF}$, $h_r$, and $h_\ell$.

\begin{lemma}[\textbf{Regularity}]\label{lemma:regularity}\mbox{}\\
	The functions $u_+,u_-,u_\mathrm{ex},u_\ell,u_r$, and $u_n$ with $n\ge4$, belong to $C^\infty(\mathbb{R}^d)$.
\end{lemma}

\begin{proof}
	We discuss the case of $u_+$ only. Define
	\begin{equation*}
	W:=V_\mathrm{DW}+\lambda w*|u_+|^2-\mu_+.
	\end{equation*}
	Then $u_+$ then solves the elliptic equation with locally Lipschitz coefficients
	\begin{equation} \label{eq:equation_u_+}
	-\Delta u_++W u_+=0.
	\end{equation}
	This means that we can apply \cite[Theorem 8.8]{GilTru-01} and deduce that $u_+\in H^2(K)$ for every compact set $K\subset \mathbb{R}^d$. In order to prove higher regularity we will use a bootstrap argument. 
	
	Recall that, for a sufficiently regular $K$,
	\begin{equation*}
	\|fg\|_{H^s(K)}\le C \|f\|_{H^{s_1}(K)} \,\|g\|_{H^{s_2}(K)}
	\end{equation*}	
	for $s<s_1+s_2-d/2$. The validity of the above formula if $K$ is replaced by $\mathbb{R}^d$ is well known, and to deduce it for compact domains one uses Stein's extension Theorem \cite[Theorem~5.24]{AdaFou-03}. Now, since $u_+\in H^2(K)$ and $W\in H^1(K)$, the above inequality proves in particular that $Wu_+\in H^1(K)$. Due to \eqref{eq:equation_u_+}, this means $\Delta u_+\in H^1(K)$, and therefore $u_+\in H^3(K)$. We can now iterate the procedure, because $u_+\in H^3(K)$ and $W\in H^1(K)$ imply $Wu\in H^2(K)$. In this way we deduce that $u_+\in H^s(K)$ for any $s>0$. This implies that $u_+$ is $C^\infty$ in any sufficiently regular compact set, which means it is $C^\infty$ on the whole of $\mathbb{R}^d$. The same argument can be repeated  for all the other functions.	
\end{proof}

\begin{lemma}[\textbf{Uniform bound for the eigenvalues}] \label{lemma:uniform_mu} \mbox{}\\
	For each $j$, the $j$-th eigenvalue $\mu_j$ of $h_\mathrm{MF}$\footnote{We are adopting the convention $\mu_1=\mu_+$, $\mu_2=\mu_-$, and $\mu_3=\mu_\mathrm{ex}$.} satisfies
	\begin{equation}
	0< \mu_j \le C_j,
	\end{equation}
	for a constant $C_j$ that does not depend on $L$.
\end{lemma}

\begin{proof}	
	First we observe that, as operators,
	\begin{equation*}
	h_\mathrm{MF} \le h_r +C,
	\end{equation*}
	where $h_r$ is the one-well Hamiltonian from \eqref{eq:single well hamil}. Indeed, $w*|u_+|^2$ and $w*|u_r|^2$ are uniformly bounded by $\sup \vert w \vert$ since $|u_+|^2$ and $|u_r|^2$ are $L^2$-normalized. Moreover $V_{\mathrm{DM}}\le V_r$ by definition. By the min-max principle, this means that the $j$-th (ordered) eigenvalue of $h_\mathrm{MF}$ is bounded by the $j$-th (ordered) eigenvalue of $h_r$ plus a constant. However, the spectrum of $h_r$ does not depend on $L$, because $h_r$ coincides with the translation (by $-\mathbf{x}$) of a fixed Hamiltonian. 
\end{proof}

We will also need an analogous result on the uniformity of Sobolev norms of the eigenvectors of $h_\mathrm{MF}$. To this end, we start with the following

\begin{lemma}[\textbf{Estimate on $h^2_\mathrm{DW}$}]\mbox{} \label{lemma:delta2}\\
	We have the quadratic form bound
	\begin{equation*}
	\frac{1}{2}(-\Delta)^2\le h_\mathrm{DW}^2+C,
	\end{equation*}
	for a constant $C$ that does not depend on $L$.
\end{lemma}

\begin{proof}
	Let $W=V_\mathrm{DW}+\lambda w*|u_+|^2$. For $\psi\in\mathcal{D}[h_\mathrm{DW}^2]$ with $\|\psi\|_{L^2}=1$ we have, after expanding the square and integrating by parts,
	\begin{equation*}
	\begin{split}
	\langle \psi,h^2_\mathrm{DW}\psi\rangle = \;&\langle\psi, \Delta^2\psi\rangle+\langle \psi,W^2\psi\rangle+\langle \nabla\psi,(\nabla W)\psi\rangle+2\langle\nabla\psi,W\nabla\psi\rangle+\langle\psi,(\nabla W)\nabla \psi\rangle\\
	\ge\;&\langle\psi, \Delta^2\psi\rangle+\langle \psi,W^2\psi\rangle+\langle \nabla\psi,(\nabla W)\psi\rangle+\langle\psi,(\nabla W)\nabla \psi\rangle,
	\end{split}
	\end{equation*}
	where for a lower bound we used $W\ge0$. By Cauchy-Schwarz we have
	\begin{equation*}
	\langle \nabla\psi,(\nabla W)\psi\rangle+\langle\psi,(\nabla W)\nabla \psi\rangle\ge -\langle\psi,(-\Delta)\psi\rangle-\langle\psi,|\nabla W|^2\psi\rangle,
	\end{equation*}
	and the further inequality $-\Delta \le \frac{1}{2}\Delta^2+2$ yields
	\begin{equation*}
	\langle \psi,h^2_\mathrm{DW}\psi\rangle\ge \frac{1}{2}\langle\psi,\Delta^2\psi\rangle+\big\langle\psi,\big(W^2-|\nabla W|^2-2\big)\psi\big\rangle.
	\end{equation*}
	The lemma follows from the claim 
	$$W^2-|\nabla W|^2\ge -C$$
	that we now prove. Let us consider the half-space $\{x_1\ge0\}$. Here,
	\begin{equation*}
	W^2(x)=|x-\bx|^{2s}+\big(\lambda w*|u_+|^2\big)^2+2\lambda |x-\bx|^sw*|u_+|^2
	\end{equation*}
	and
	\begin{equation*}
	|\nabla W(x)|^2=s^2|x-\bx|^{2s-2}+\big| 2\lambda  w*(u_+\nabla u_+)\big|^2+4\lambda s|x-x_N|^{s-1}\frac{x-\bx}{|x-\bx|}\cdot w*(u_+\nabla u_+).
	\end{equation*}
	Let us consider the difference $W^2-|\nabla W|^2$. For $W^2$ we will use the estimate
	\begin{equation*}
	W^2(x)\ge |x-\bx|^{2s}.
	\end{equation*}
	For the $\lambda^2$-term in $|\nabla W|^2$ we have, by Young's and H\"older's inequalities,
	\begin{equation*}
	-\big( 2\lambda  w*(u_+\nabla u_+)\big)^2\ge -4\lambda^2\|w\|_{L^\infty}^2\, \|u_+\|_{L^2}^2\,\|u_+\|_{H^1}^2 \ge -C,
	\end{equation*}
	because $\|u_+\|_{H^1}^2\le \mu_+ \le C$ by Lemma \ref{lemma:uniform_mu}. 	For the $\lambda$-term in $|\nabla W|^2$ we use Cauchy-Schwarz followed by Young and H\"{o}lder inequalities to get
	\begin{equation*}
	-4\lambda s|x-x_N|^{s-1}\frac{x-\bx}{|x-\bx|}\cdot w*(u_+\nabla u_+) \ge - \delta |x-\bx|^{2s-2}-C_\delta.
	\end{equation*}
	The three last inequalities imply
	\begin{equation*}
	W^2(x)-|\nabla W(x)|^2 \ge |x-\bx|^{2s} - (s^2+\delta)|x-\bx|^{2s-2}-C_\delta-C \ge -C'.
	\end{equation*}
	This concludes the proof.
\end{proof}

\begin{lemma}[\textbf{Uniform Sobolev regularity}] \label{lemma:uniform_Sobolev}\mbox{}\\
	 For each $j$, the $j$-th (normalized) eigenvector of $h_\mathrm{MF}$ 
	 satisfies
	\begin{equation}
	\|u_j\|_{H^2(\mathbb{R}^d)}\le C_j,\qquad \|u_j\|_{L^\infty(\mathbb{R}^d)}\le C_j,
	\end{equation}
	for a constant $C_j$ that does not depend on $L$.
\end{lemma}
\begin{proof}
By the Sobolev embedding
\begin{equation*}
\|f\|_{L^\infty(\mathbb{R}^d)}\le C \|f\|_{H^2(\mathbb{R}^d)},
\end{equation*}
that holds for $d=1,2,3$, the second inequality follows from the first one. To prove the first one, we recognize that
\begin{equation*}
\|u_j\|_{H^2}^2 = \| -\Delta u_j\|_{L^2}^2+\|u_j\|_{L^2}^2 \le \langle u_j,h_{\mathrm{MF}} ^2u_j\rangle +C,
\end{equation*}
where the second inequality follows from Lemma \ref{lemma:delta2}. Since $u_j$ is an eigenvector of $h_\mathrm{MF}$, we have
\begin{equation*}
\|u_j\|_{H^2}^2 \le \mu_j^2 +C \le C_j,
\end{equation*}
thanks to Lemma \ref{lemma:uniform_mu}.
\end{proof}

\subsection{Decay estimates}\label{sec:decay}

Key ingredients for all our estimates are the following estimates on the decay of the eigenvectors of $h_\mathrm{DW}$ far from $\mathbf{x}$ and $-\mathbf{x}$.

\begin{proposition}[\textbf{Decay estimates for eigenmodes}]\label{prop:decay_estimates}\mbox{}\\ 
Let $\varepsilon$ be an arbitrarily small parameter and $A_{\mathrm{DW}}$ be as in~\eqref{eq:A_DW}.
	\begin{itemize}
		\item\textbf{Integral bound.}
		Let $u_j$ be any eigenvector of $h_\mathrm{DW}$, corresponding to eigenvalue $\mu_j$ and with normalization $\|u_j\|_{L^2}=1$. There exists $C_{\varepsilon,j}>0$ (depending on $\varepsilon$ and $j$ but not on $L$) such that
		\begin{equation} \label{eq:integral_decay_bound}
		\int_{ \mathbb{R}^d} \left| \nabla \left(e^{(1-\varepsilon)A_\mathrm{DW}}u_j\right) \right|^2\,dx+	\int_{ \mathbb{R}^d}   e^{2(1-\varepsilon)A_\mathrm{DW}}| u_j |^2\,dx\le C_{\varepsilon,j}.
		\end{equation}
		\item\textbf{Pointwise bound.} Let $u_j$ be as above. For $R>0$ large enough, there exists $C_{\varepsilon,j}>0$ such that
		\begin{equation} \label{eq:pointwise_upper_bound}
		|u_j(x)|\le C_{\varepsilon,j} e^{-(1-\varepsilon)A_\mathrm{DW}(x)}
		\end{equation}
		for every $x$ such that $|x-\mathbf{x}|,|x+\mathbf{x}|\ge R$.
		\item\textbf{Pointwise bound for one-well modes.} For $R>0$ large enough, there exists $C_\varepsilon$ such that
		\begin{align} \label{eq:decay_u_r}
		u_{r}(x)\le\;& C_\varepsilon e^{-(1-\varepsilon)A_{}(x-\bx)}
		\end{align}
		for $|x-\mathbf{x}|\ge R$ and
		\begin{align}
		\label{eq:decay_u_ell}
		u_{\ell}(x)\le\;& C_\varepsilon e^{-(1-\varepsilon)A_{}(x+\bx)}
		\end{align}
		for $|x+\mathbf{x}|\ge R$.
	\end{itemize}
\end{proposition}

We start the proof with the following Lemma.

\begin{lemma}[\textbf{Integral decay bounds}]\label{lemma:first_agmon_lemma}\mbox{}\\
	Let $\Phi$ be locally Lipschitz and let its gradient be defined as the $L^\infty$ limit of a mollified sequence $\nabla \Phi_\varepsilon$. Let moreover $u$ be an eigenvector of $h_\mathrm{DW}$, corresponding to the eigenvalue $\mu$ and with normalization $\|u\|_{L^2}=1$. Define the total potential $W:= V_\mathrm{DW}+\lambda w*|u_+|^2-\mu$ and assume that $W \geq |\nabla \Phi|^2$ outside of a compact set. Then
	\begin{equation}
	\int_{\mathbb{R}^d} \left| \nabla \left( e^{\Phi} u\right) \right|^2\,dx+ \int_{ \mathbb{R}^d} \left( W-\left| \nabla \Phi\right|^2\right)e^{2\Phi}|u|^2\,dx \leq 0.
	\end{equation}
\end{lemma}

This is very much in the spirit of~\cite[Theorem~1.5]{Agmon-82}, \cite[Lemma 2.3]{HelSjo-85} or \cite[Theorem 3.1.1]{Helffer-88}.

\begin{proof}
	To avoid a number of boundary terms that would arise after integration by parts, let us introduce the sequence of smooth localization functions
	\begin{equation*}
	\theta_j(x)=\begin{cases}
	1 \quad& |x|\leq j\\
	0 \quad& |x| \geq 2j\\
	\frac{e^{\frac{1}{1-|x|/j}}}{e^{\frac{1}{1-|x|/j}}+e^{\frac{j}{|x|}}}\quad& j\le |x| \le 2j.
	\end{cases}
	\end{equation*}
	Since $\theta_j$ is obtained by dilating a fixed function by a factor $j$, we have
	\begin{equation*}
	\left\|\nabla \theta_j\right\|_{L^\infty(\mathbb{R}^d)}\le \frac{C}{j},\quad\left\|\Delta \theta_j\right\|_{L^\infty(\mathbb{R}^d)} \le \frac{C}{j^2}.
	\end{equation*}
	The localization function $\theta_j$ will tend to one pointwise as $j\to\infty$, thus yielding integrals on the whole $\mathbb{R}^d$, while the terms depending on its derivatives will vanish thanks to the above bounds. We further define, for $k\in\mathbb{N}$,
	\begin{equation*}
	\Phi_k(x):=\min\{\Phi(x),k\}.
	\end{equation*}
	This is a uniformly Lipschitz version of $\Phi_k$, which tends to $\Phi$ as $k\to\infty$.
	
	By definition the function $u$ satisfies
	\begin{equation*}
	\left( -\Delta+W\right)u=0.
	\end{equation*}
	Recall that we fixed all phases so as to have only real-valued eigenvectors of $h_\mathrm{DW}$, so $u$ is real-valued.
	We multiply the above equation by $\theta_j e^{2\Phi_k}u$, and integrate by parts. We get
	\begin{equation*}
	\begin{split}
	0=\;&\int_{\mathbb{R}^d} \theta_j e^{\Phi_k} \nabla \left( e^{\Phi_k}u\right)\cdot \nabla u\,dx+ \int_{\mathbb{R}^d} \theta_j e^{2\Phi_k}u\nabla \Phi_k\cdot \nabla u\,dx \\
	&+\int_{ \mathbb{R}^d} e^{2\Phi_k} u \nabla \theta_j\cdot \nabla u\,dx+\int_{ \mathbb{R}^d}\theta_je^{2\Phi_k}|u|^2W\,dx.
	\end{split}
	\end{equation*}
	By using Leibniz rule in the first term in the right hand side we get
	\begin{equation*}
	\begin{split}
	0=\;&\int_{\mathbb{R}^d} \theta_j \left|\nabla \left( e^{\Phi_k}u\right)\right|^2\,dx -\int_{\mathbb{R}^d} \theta_j e^{\Phi_k} u \nabla \left( e^{\Phi_k}u\right)\cdot \nabla \Phi_k\,dx+\int_{\mathbb{R}^d} \theta_j e^{2\Phi_k}u\nabla \Phi_k\cdot \nabla u\,dx \\
	&+\int_{ \mathbb{R}^d} e^{2\Phi_k} u \nabla \theta_j\cdot \nabla u\,dx+\int_{ \mathbb{R}^d}\theta_je^{2\Phi_k}|u|^2W\,dx\\
	=\;&\int_{\mathbb{R}^d} \theta_j \left|\nabla \left( e^{\Phi_k}u\right)\right|^2\,dx -\int_{\mathbb{R}^d} \theta_j e^{2\Phi_k} u^2 \left|\nabla \Phi_k\right|^2\,dx\\
	&+\int_{ \mathbb{R}^d} e^{2\Phi_k} u \nabla \theta_j\cdot \nabla u\,dx+\int_{ \mathbb{R}^d}\theta_je^{2\Phi_k}|u|^2W\,dx,
	\end{split}
	\end{equation*}
	which is rewritten as
	\begin{equation*}
	\int_{\mathbb{R}^d} \theta_j \left|\nabla \left( e^{\Phi_k}u\right)\right|^2\,dx + \int_{ \mathbb{R}^d}\theta_j\left(W-\left|\nabla \Phi_k\right|^2\right)e^{2\Phi_k}|u|^2\,dx = \int_{ \mathbb{R}^d} e^{2\Phi_k} u \nabla \theta_j\cdot \nabla u\,dx.
	\end{equation*}
	We need to show that the term in the right hand side converges to zero as $j\to\infty$, and that the quantity in the left hand side is controlled when $j\to\infty$ followed by $k\to\infty$.
	
	First,
	\begin{equation*}
	\left|\int_{ \mathbb{R}^d} e^{2\Phi_k} u \nabla \theta_j\cdot \nabla u\,dx \right| \le C_k \|\nabla \theta_j\|_{L^\infty(\mathbb{R}^d)}\int_{ \mathbb{R}^d} |\nabla u|\,|u|\,dx \le \frac{C_k}{j} \|u\|_{L^2}\|u\|_{H^1}.
	\end{equation*}
	This implies
	\begin{equation*}
	\lim_{j\to\infty}\int_{\mathbb{R}^d} \theta_j \left|\nabla \left( e^{\Phi_k}u\right)\right|^2\,dx + \int_{ \mathbb{R}^d}\theta_j\left(W-\left|\nabla \Phi_k\right|^2\right)e^{2\Phi_k}|u|^2\,dx=0,
	\end{equation*}
	which, by monotone convergence, means
	\begin{equation*}
	\int_{\mathbb{R}^d} \left|\nabla \left( e^{\Phi_k}u\right)\right|^2 \,dx+ \int_{ \mathbb{R}^d}\left(W-\left|\nabla \Phi_k\right|^2\right)e^{2\Phi_k}|u|^2\,dx=0.
	\end{equation*}
	Since $\Phi_k\to\Phi$ pointwise as $k\to\infty$, Fatou's Lemma yields the desired result.
\end{proof}

We are now ready to provide the

\begin{proof}[Proof of Proposition \ref{prop:decay_estimates}]
	We will apply Lemma \ref{lemma:first_agmon_lemma} and show how to recover \eqref{eq:integral_decay_bound}. We fix a constant $\kappa>0$, and consider the set
	\begin{equation*}
	\begin{split}
	\Omega_\kappa:=\;&\{ x\in\mathbb{R}^d\;|\; \left(2\varepsilon-\varepsilon^2\right)V_\mathrm{DW}-\mu_j \ge \kappa \}\\
	\equiv \;&\left\{ x\in\mathbb{R}^d\;|\;\begin{cases}
	|x-\mathbf{x}|\ge \left( (\kappa+\mu_j)/(2\varepsilon-\varepsilon^2) \right)^{1/s},\quad x_1\ge0\\
	|x+\mathbf{x}|\ge \left( (\kappa+\mu_j)/(2\varepsilon-\varepsilon^2) \right)^{1/s},\quad x_1\le0
	\end{cases} \right\}\\
	 \Omega_\kappa^c=\;&\mathbb{R}^d\setminus \Omega_\kappa.
	\end{split}
	\end{equation*}
	We also define, for any $\varepsilon<1/2$,
	\begin{equation*}
	\Phi=(1-\varepsilon)A_\mathrm{DW},
	\end{equation*}
	so that
	\begin{equation*}
	|\nabla \Phi|^2=(1-\varepsilon)^2V_\mathrm{DW}.
	\end{equation*}
	Notice that the function $(2\varepsilon-\varepsilon^2)V_\mathrm{DW}-\mu_j$ appearing in the definition of $\Omega_\kappa$ is smaller than $W-|\nabla\Phi|^2$, where $W=V_\mathrm{DW}+\lambda w*|u_+|^2-\mu_j$. As a consequence, $W\ge\kappa$ on the whole $\Omega_\kappa$.
	
	We thus have
	\begin{equation*}
	\begin{split}
	\int_{\mathbb{R}^d}& \left| \nabla \left( e^{\Phi} u_j\right) \right|^2\,dx+ \kappa \int_{ \Omega_k} e^{2\Phi}|u_j|^2\,dx \\
	\leq\;& \int_{\mathbb{R}^d} \left| \nabla \left( e^{\Phi} u_j\right) \right|^2\,dx+  \int_{ \Omega_k} \left( W-|\nabla \Phi|^2 \right)e^{2\Phi}|u_j|^2\,dx\\
	\le\;&\int_{\mathbb{R}^d} \left| \nabla \left( e^{\Phi} u_j\right) \right|^2\,dx+  \int_{ \mathbb{R}^d} \left( W-|\nabla \Phi|^2 \right)e^{2\Phi}|u_j|^2\,dx-\int_{ \Omega_k^c} \left( W-|\nabla \Phi|^2 \right)e^{2\Phi}|u_j|^2\,dx,
	\end{split}
	\end{equation*}
	and, by Lemma \ref{lemma:first_agmon_lemma} and the inequality $W-|\nabla\Phi|^2\ge -\mu_j$, we get
	\begin{equation*}
	\begin{split}
	\int_{\mathbb{R}^d} \left| \nabla \left( e^{\Phi} u_j\right) \right|^2\,dx+ \kappa \int_{ \Omega_k} e^{2\Phi}|u_j|^2 \,dx\le\;&\mu_j \int_{ \Omega_k^c} e^{2\Phi}|u_j|^2\,dx.
	\end{split}
	\end{equation*}
	This in turn implies
	\begin{equation*}
	\begin{split}
	\int_{\mathbb{R}^d} \left| \nabla \left( e^{\Phi} u_j\right) \right|^2\,dx+ \kappa \int_{ \mathbb{R}^d} e^{2\Phi}|u_j|^2\,dx \le\;&\left(\mu_j+\kappa \right) \int_{ \Omega_k^c} e^{2\Phi}|u_j|^2\,dx.
	\end{split}	
	\end{equation*}
	Since $e^{2\Phi}$ is easily seen to be bounded by a $\varepsilon$-dependent constant on $\Omega_\kappa^c$, the above inequality implies \eqref{eq:integral_decay_bound} after choosing, for example, $\kappa=1$. 
	
	Let us now prove \eqref{eq:pointwise_upper_bound}. The function $|u_j|^2$ satisfies
	\begin{equation*}
	\Delta |u_j|^2=2|\nabla u_j|^2+ 2\left(V_\mathrm{DW}+\lambda w*|u_+|^2 -\mu_j\right)|u_j|^2,
	\end{equation*}
	and therefore
	\begin{equation*}
	\left(\Delta|u_j|^2\right)(x) \ge 0
	\end{equation*}
	for $|x-\mathbf{x}|,|x+\mathbf{x}|\ge R$ with $R$ large enough. Hence, by the mean value property for subharmonic functions (see, e.g., \cite[Theorem 2.1]{GilTru-01} or~\cite[Section~9.3]{LieLos-01}),
	\begin{equation*}
	|u_j(x)|^2\le \frac{1}{\mathrm{Vol}B_x(1)}\int_{B_x(1)} |u_j(y)|^2\,dy,
	\end{equation*}
	where $B_x(1)$ is the ball of radius $1$ centered at $x$ (assume $R$ is large enough so that $u_j$ is subharmonic on the whole $B_x(1)$).
	We multiply and divide by $e^{2(1-\varepsilon)A_\mathrm{DW}}$ inside the integral and use the Taylor-like expansion
	\begin{equation*}
	\left| e^{-2(1-\varepsilon)A_\mathrm{DW}(y)}-e^{-2(1-\varepsilon)A_\mathrm{DW}(x)}\right|\le Ce^{-2(1-\varepsilon')A_\mathrm{DW}(x)}  \qquad y\in B_x(1),
	\end{equation*}
	which holds for some $\varepsilon<\varepsilon'<1/2$ used to absorb the growth of the gradient.
	This gives
	\begin{equation*}
	|u_j(x)|^2\le C e^{-2(1-\varepsilon')A_\mathrm{DW}(x)}\int_{B_x(1)} |u_j(y)|^2 e^{2(1-\varepsilon)A_\mathrm{DW}(y)}\,dy.
	\end{equation*}
	The remaining integral is bounded by a constant thanks to \eqref{eq:integral_decay_bound}, and we thus get
	\begin{equation*}
	|u_j(x)|\le C_{\varepsilon,u}e^{-(1-\varepsilon')A_\mathrm{DW}(x)}.
	\end{equation*}
	This is precisely of the form \eqref{eq:pointwise_upper_bound} after a redefinition of the constants. The bounds for $u_r$ and $u_\ell$ are obtained through similar arguments, which we do not reproduce.
\end{proof}

The above allows to efficiently bound most terms that have to do with the tunneling effect in the sequel. A list of such is provided in Appendix~\ref{sect:app}.

\subsection{First approximations}


An important ingredient for the sequel is a first approximation of $u_+$ in terms of functions localized in the left or right wells:

\begin{proposition}[\textbf{First properties of $u_+$ and $u_-$.}]\label{prop:first_convergence}\mbox{}\\
Let $\chi_{x_1\ge0},\chi_{x_1\le0}$ be a smooth partition of unity such that
	\begin{equation*}
	\begin{split}
	\chi_{x_1\ge0}^2+\chi_{x_1\le0}^2=1,& \qquad\qquad\qquad \chi_{x_1\ge0}(x)=\chi_{x_1\le0}(-x_1,x_2,\dots,x_d),\\
	\chi_{x_1\ge0}(x)=0\quad\text{on }\{x_1\le -C\},&\qquad\qquad\qquad \|\nabla \chi_{x_1\ge0}\|_\infty\le C.
	\end{split}
	\end{equation*}
	Then, with $u_\ell$ and $u_r$ the left and right Hartree minimizers solving~\eqref{eq:Hartree left right}, and $T$ the tunneling parameter \eqref{eq:tunnel},
	\begin{equation} \label{eq:l_r_L2_norm_u+}
	\begin{split}
	\big\|\chi_{x_1\ge0} u_+-u_r\big\|_{L^2}\le\;&  C_\varepsilon T^{1/2 -\varepsilon} \\ 
	\big\|\chi_{x_1\le0} u_+-u_\ell\big\|_{L^2} \le\;& C_\varepsilon T^{1/2 -\varepsilon},
	\end{split}
	\end{equation}
	and, for an appropriate choice of the phase of $u_-$,
	\begin{equation} \label{eq:l_r_L2_norm_u-}
	\begin{split}
	\big\|\chi_{x_1\ge0} u_--u_r\big\|_{L^2}\le\;&  C_\varepsilon T^{1/4 -\varepsilon} \\ 
	\big\|\chi_{x_1\le0} u_- +u_\ell\big\|_{L^2} \le\;& C_\varepsilon T^{1/4 -\varepsilon}.
	\end{split}
	\end{equation}
\end{proposition}

The approximation \eqref{eq:l_r_L2_norm_u-} has a worse rate than \eqref{eq:l_r_L2_norm_u+}, and therefore it does not allow to directly deduce \eqref{eq:L2_convergence} with the desired rate yet. It follows from the above and standard elliptic regularity estimates that 
\begin{align}\label{eq:strong CV}
 \chi_{x_1\ge0} u_+ (x-\bx) &\underset{L\to\infty}{\to} u_0 \nonumber\\
\chi_{x_1\le 0} u_+ (x + \bx) &\underset{L\to\infty}{\to} u_0 \nonumber\\
\chi_{x_1\ge0} u_- (x-\bx) &\underset{L\to\infty}{\to} u_0 \nonumber\\
\chi_{x_1\le0} u_- (x+\bx) &\underset{L\to\infty}{\to} - u_0 ,
\end{align}
where $u_0$ is the minimizer of the one-well Hartree functional (obtained by setting $\bx= 0$ in the definition~\eqref{eq:left right Hartree} of the left and right well functionals). For any fixed smooth bounded set $\Omega$ the convergence is strong in any Sobolev space $H^s(\Omega)$ and (by Sobolev embeddings) in any H\"older space $\cC^{\alpha} (\Omega)$. Indeed, repeatedly differentiating the elliptic PDEs satisfied by $u_+$ and $u_-$ one obtains uniform bounds in any Sobolev space. Using local compact embeddings one can then extract convergent subsequences in these spaces, and the above Proposition uniquely identifies the limit.
 
 We will prove Proposition \ref{prop:first_convergence} using energy inequalities, which requires the following two lemmas.
 
\begin{lemma}\textbf{Stability inequality for gapped Hamiltonians.}\label{lemma:stability_generic}\mbox{}\\ 
	Let $h$ be a self-adjoint Hamiltonian with compact resolvent on a Hilbert space $\mathcal{H}$. Let $\lambda_0$ be the ground state energy with ground state $u_0$ (with $\|u_0\|_\mathcal{H}=1$), and let $G\ge0$ be the gap between ground state and first excited state. Then, for any $u\in\mathcal{D}(h)$ with $\|u\|_\mathcal{H}=1$,
	\begin{equation} \label{eq:stability_h}
	\langle u,hu\rangle_{\mathcal{H}} \ge \lambda_0+\frac{G}{2}\min_{\theta\in[0,2\pi]}\Big\|e^{i\theta}u-u_0\Big\|_{\mathcal{H}}^2.	\end{equation}
\end{lemma}
The assumption of compactness of the resolvent is clearly not crucial for this Lemma. We anyway keep it since that is the only case we will encounter. We also remark that in the case $G=0$, i.e., a degenerate ground state, the statement is trivial.
\begin{proof}
	By the assumptions we have the decomposition
	\begin{equation*}
	h=\lambda_0 \ket{u_0}\bra{u_0}+\sum_{n}\lambda_n\ket{u_n}\bra{u_n}
	\end{equation*}
	with $\lambda_n\ge \lambda_0+G$ for every $n$. Then
	\begin{equation*}
	\langle u,hu\rangle_{\mathcal{H}} \ge\lambda_0 \big|\langle u_0,u\rangle_{\mathcal{H}}\big|^2+(\lambda_0+G)\big(1-\big|\langle u_0,u\rangle_{\mathcal{H}}\big|^2\big)=\lambda_0+G\big(1-\big|\langle u_0,u\rangle_{\mathcal{H}}\big|^2\big).
	\end{equation*}
	On the other hand we have
	\begin{equation*}
	\min_{\theta\in[0,2\pi]}\big|e^{i\theta}u-u_0\big|_{\mathcal{H}}^2=2-2\max_{\theta\in[0,2\pi]}\mathrm{Re}\Big(e^{i\theta}\langle u_0,u\rangle_{\mathcal{H}}\Big)=2-2\big|\langle u_0,u\rangle_{\mathcal{H}}\big|.
	\end{equation*}
	Since $1-|\langle u_0,u\rangle_{\mathcal{H}}|\le 1-|\langle u_0,u\rangle_{\mathcal{H}}|^2$, the last two equations yield the desired estimate.
\end{proof}

%
%
\begin{lemma}[\textbf{Stability inequality for the one-well Hartree functionals.}]\mbox{}\label{lemma:stability_hartree}\\ 
	For a generic $u\in H^2(\mathbb{R}^d)\cap L^2\big(\mathbb{R}^d,V_{\mathrm{r}}(x)dx\big)$ with $\|u\|_{L^2}^2=\frac{1}{2}$, the following stability inequality holds:
	\begin{align}
	\label{eq:stability_L2_functional_r}
	\mathcal{E}_r[u]\ge\;& \mathcal{E}_r[u_r]+C\min_{\theta\in[0,2\pi]}\big\|e^{i\theta}u-u_r\big\|_{L^2}^2.
	\end{align}
\end{lemma}
\begin{proof}
	First, let us notice that an application of Lemma \ref{lemma:stability_generic} for $h=h_r$ yields (the different normalization has no effect since the whole inequality is homogeneous with respect to an overall factor)
	\begin{equation}\label{eq:stability_L2_form}
	\langle u,h_r u\rangle\ge\; \langle u_r,h_r u_r\rangle+C\min_{\theta\in[0,2\pi]}\big\|e^{i\theta} u-u_r\big\|_{L^2}^2.
	\end{equation}
	Indeed, $h_r$ is obtained from a $L$-independent Hamiltonian by a translation by $\bx$, and therefore its spectrum does not depend on $L$. The properties of $V$ and $w$ imply that $h_r$ must have a gap (see, e.g., \cite[Theorem XIII.47]{ReeSim4}). The $L$-independence of the spectrum implies that such a gap does not depend on $L$. We can therefore apply Lemma \ref{lemma:stability_generic} and consider $G/2$ as a fixed constant. 
	%
	%
	
	To deduce~\eqref{eq:stability_L2_functional_r}, a simple computation gives
	\begin{equation*}
	\begin{split}
	\mathcal{E}_r[u]-\mathcal{E}_r[u_r]=\;&\langle u, h_r u\rangle-\langle u_r,h_r u_r\rangle+\frac{\lambda}{2}\int_{\mathbb{R}^d}(w*|u_r|^2)|u_r|^2\,dx\\
	&-\frac{\lambda}{2}\int_{\mathbb{R}^d}(w*|u|^2)|u|^2\,dx+\lambda\int_{\mathbb{R}^d}\big(w*(|u|^2-|u_r|^2)\big)|u|^2\,dx\\
	=\;&\langle u, h_\ell u\rangle-\langle u_r,h_\ell u_r\rangle+\frac{\lambda}{2}\int_{\mathbb{R}^d}\big(w*(|u_r|^2-|u|^2)\big)\big(|u_r|^2-|u|^2\big)\,dx.
	\end{split}
	\end{equation*}
	Since $\widehat w\ge0$, the last integral on the right hand side is non-negative. We discard it for a lower bound and get
	\begin{equation} \label{eq:stability_functional_form}
	\mathcal{E}_r[u]-\mathcal{E}_r[u_r]\ge \langle u, h_r u\rangle-\langle u_r,h_r u_r\rangle,
	\end{equation}
	which proves \eqref{eq:stability_L2_functional_r} thanks to \eqref{eq:stability_L2_form}.
\end{proof}

We are now ready to give the 

\begin{proof}[Proof of Proposition \ref{prop:first_convergence}]
	Let us first show the upper bound
	\begin{equation} \label{eq:upper_bound_u+}
	\mathcal{E}_\mathrm{DW}[u_+]\le 2\mathcal{E}_r[u_r]+C_\varepsilon T^{1-\varepsilon}.
	\end{equation}
	The normalized state $(u_r+u_\ell)/\|u_r+u_\ell\|_{L^2}$ is an admissible trial function for the minimization of $\mathcal{E}_\mathrm{DW}$ at unit mass. Notice that, by positivity of $u_r$ and $u_\ell$,
	\begin{equation*}
	\|u_r+u_\ell\|_{L^2}^2=1+2\mathrm{Re}\langle u_r,u_\ell\rangle\ge 1,
	\end{equation*}
	and hence we can ignore the norms in the denominator for an upper bound. We have
	\begin{equation} \label{eq:expansion_upper_bound}
	\begin{split}
	\mathcal{E}_\mathrm{DW}[u_+]\le\;& \mathcal{E}_\mathrm{DW}\Big[\frac{u_r+u_\ell}{\|u_r+u_\ell\|_{L^2}}\Big]\\
	\le\;&\int_{\mathbb{R}^d}|\nabla u_r|^2\,dx+\int_{\mathbb{R}^d} V_\mathrm{DW}|u_r|^2\,dx+\frac{\lambda}{2}\int_{\mathbb{R}^d} |u_r|^2 (w*|u_r|^2)\,dx\\
	&+\int_{\mathbb{R}^d}|\nabla u_\ell|^2\,dx+\int_{\mathbb{R}^d} V_\mathrm{DW}|u_\ell|^2\,dx+\frac{\lambda}{2}\int_{\mathbb{R}^d}|u_\ell|^2 (w*|u_\ell|^2)\,dx\\
	&+2\int_{\mathbb{R}^d}\nabla u_\ell\cdot \nabla u_r\,dx+2\int_{\mathbb{R}^d} V_\mathrm{DW}u_\ell u_r\,dx\\
	&+\frac{\lambda}{2}\iint_{\mathbb{R}^d\times\mathbb{R}^d} w(x-y)\Big[2u_\ell(x)u_r(x)\big|u_r(y)+u_\ell(y)\big|^2\\
	&\qquad\qquad\qquad\qquad\qquad+2|u_\ell(x)|^2\big(|u_r(y)|^2+2u_\ell(y)u_r(y)\big)\Big]\,dxdy.
	\end{split}
	\end{equation}
	In the first two lines in the right hand side we can use, respectively, $V_\mathrm{DW}\le V_\ell$ and $V_\mathrm{DW}\le V_r$. In this way the first line equals $\mathcal{E}_r[u_r]$ and the second one equals $\mathcal{E}_\ell[u_\ell]$, which actually coincide by translation invariance.
	The terms in the third line are remainders as follows from \eqref{eq:nabla_nabla_term} and \eqref{eq:V_l_r_term}. We then deduce
	\begin{equation*}
	\begin{split}
	\mathcal{E}_\mathrm{DW}[u_+]\le\;&2\mathcal{E}_r[u_r]+C_\varepsilon T^{1-\varepsilon}\\
	&+\frac{\lambda}{2}\iint_{\mathbb{R}^d\times\mathbb{R}^d} w(x-y)\Big[2u_\ell(x)u_r(x)\big|u_\ell(y)+u_r(y)\big|^2\\
	&\qquad\qquad\qquad\quad\qquad\qquad+2|u_\ell(x)|^2\big(2u_\ell(y)u_r(y)+|u_r(y)|^2\big)\Big]\,dxdy.
	\end{split}
	\end{equation*}
	To get rid of the last terms, involving $w$, we notice that by the Cauchy-Schwarz and Young inequalities we have
	\begin{equation*}
	\bigg|\iint_{\mathbb{R}^d\times\mathbb{R}^d} w(x-y) u_\ell(x)u_r(x) g(y)\,dxdy\,\bigg|\le \|w\|_{L^\infty}\,\|g\|_{L^1}\,\int_{\mathbb{R}^d} u_\ell u_r\,dx,
	\end{equation*}
	and the scalar product on the right is estimated using \eqref{eq:scalar_product}. The only remaining term to estimate is
	\begin{equation*}
	\lambda\iint_{\mathbb{R}^d\times\mathbb{R}^d}w(x-y)|u_\ell(x)|^2|u_r(y)|^2dxdy,
	\end{equation*}
	which we bound using~\eqref{eq:convol_two_sides}. We thus precisely obtain~\eqref{eq:upper_bound_u+}.
	
	Let us now prove the lower bound
	\begin{equation} \label{eq:lower_bound_u+}
	\mathcal{E}_\mathrm{DW}[u_+]\ge 2\mathcal{E}_r\big[\chi_{x_1\ge0} u_+\big]-C_\varepsilon T^{1-\varepsilon}.
	\end{equation}
	Using the IMS localization formula we have
	\begin{equation*}
	-\Delta+V_\mathrm{DW}=-\chi_{x_1\ge0}\Delta\chi_{x_1\ge0}-\chi_{x_1\le0}\Delta\chi_{x_1\le0}+ V_\mathrm{DW}\chi_{x_1\ge0}^2+V_\mathrm{DW}\chi_{x_1\le0}^2-(\nabla\chi_{x_1\ge0})^2- (\nabla\chi_{x_1\le0})^2.
	\end{equation*}
	Moreover,
	\begin{equation*}
	\begin{split}
	\int_{\mathbb{R}^d} (w*|u_+|^2)|u_+|^2\,dx=\;&\int_{\mathbb{R}^d} (w*|\chi_{x_1\ge0} u_+|^2)|\chi_{x_1\ge0} u_+|^2\,dx\\
	&+\int_{\mathbb{R}^d} (w*|\chi_{x_1\le0} u_+|^2)|\chi_{x_1\le0} u_+|^2\,dx\\
	&+2\int_{\mathbb{R}^d} (w*|\chi_{x_1\le0} u_+|^2)|\chi_{x_1\ge0} u_+|^2\,dx.
	\end{split}
	\end{equation*}
	The last summand in the right hand side of the last equation is positive and we will simply discard it for a lower bound. We thus have
	\begin{equation*}
	\begin{split}
	\mathcal{E}_\mathrm{DW}[u_+]\ge\;& 2\mathcal{E}_r\big[\chi_{x_1\ge0} u_+\big]+\int_{\mathbb{R}^d}\big(V_\mathrm{DW}-V_r)\chi_{x_1\ge0}^2|u_+|^2\,dx+\int_{\mathbb{R}^d}\big(V_\mathrm{DW}-V_\ell)\chi_{x_1\le0}^2|u_+|^2\,dx\\
	&-\int_{\mathbb{R}^d}(\nabla\chi_{x_1\ge0})^2|u_+|^2\,dx-\int_{\mathbb{R}^d}(\nabla\chi_{x_1\le0})^2|u_+|^2\,dx.
	\end{split}
	\end{equation*}
	The first two integrals in the right hand side are estimated in \eqref{eq:V_DW_to_V}. The integrals in the second line are smaller or equal than the quantities estimated in \eqref{eq:localized_without_anything}, because $|\nabla \chi_{x_1\ge0}|$ and $|\nabla \chi_{x_1\le0}|$ are by construction bounded by a constant and localized in $\{-C\le x_1\le C\}$. This proves \eqref{eq:lower_bound_u+}.

	Comparing \eqref{eq:upper_bound_u+} and \eqref{eq:lower_bound_u+} we deduce
	\begin{equation*}
	\mathcal{E}_r\big[\chi_{x_1\ge0} u_+\big]\le \mathcal{E}_r [u_r]+C_\varepsilon T^{1-\varepsilon}.
	\end{equation*}
	On the other hand, a direct application of Lemma \ref{lemma:stability_hartree} with $u=\chi_{x_1\ge0} u_+$ (notice that the property $\chi_{x_1\ge0}^2+\chi_{x_1\le0}^2=1$ together with the symmetry of $u_+$ imply $\|u\|_{L^2}^2=1/2$) yields
	\begin{equation*}
	\mathcal{E}_r\big[\chi_{x_1\ge0} u_+\big]\ge \mathcal{E}_r[u_r]+C\min_{\theta\in[0,2\pi]}\big\|e^{i\theta}\chi_{x_1\ge0} u_+-u_r\big\|_{L^2}^2=\mathcal{E}_r[u_r]+C\big\|\chi_{x_1\ge0} u_+-u_r\big\|_{L^2}^2,
	\end{equation*}
	having noticed that the minimum is attained at $\theta=0$ since $\chi_{x_1\ge0} u_+$ and $u_r$ are positive. The last two formulae imply the first estimate in \eqref{eq:l_r_L2_norm_u+}. The second one immediately follows since $u_+$ is symmetric under reflection across the $x_1=0$ axis, while $u_r$ is mapped into $u_\ell$ by such a reflection.
	
	Let us now prove \eqref{eq:l_r_L2_norm_u-}. As a first ingredient, let us show the following inequality:
	\begin{equation} \label{eq:difference_mf_potentials}
	\begin{split}
	\Big|\iint_{\mathbb{R}^d\times\mathbb{R}^d}w(x-y)|u_r(x)|^2\Big(|u_+(y)|^2-|u_{r}(y)|^2\Big)\,dxdy\Big|\le C_\varepsilon T^{1/2-\varepsilon}.
	\end{split}
	\end{equation}
	Clearly, an analogous inequality holds if $u_r$ is replaced by $u_\ell$. To prove \eqref{eq:difference_mf_potentials}, let us decompose
	\begin{equation*}
	\begin{split}
	\iint_{\mathbb{R}^d\times\mathbb{R}^d}&w(x-y)|u_r(x)|^2\Big(|u_+(y)|^2-|u_r(y)|^2\Big)\,dxdy\\
	=\;&\iint_{\mathbb{R}^d\times\mathbb{R}^d}w(x-y)|u_r(x)|^2\Big(|\chi_{x_1\ge0}(y)u_+(y)|^2-|u_r(y)|^2\Big)\,dxdy\\
	&+\iint_{\mathbb{R}^d\times\mathbb{R}^d}w(x-y)|u_r(x)|^2|\chi_{x_1\le0}(y)u_+(y)|^2\,dxdy\\
	=\;&\iint_{\mathbb{R}^d\times\mathbb{R}^d}w(x-y)|u_r(x)|^2\Big(\chi_{x_1\ge0}(y)u_+(y)-u_r(y)\Big)\Big(\chi_{x_1\ge0}(y)u_+(y)+u_r(y)\Big)\,dxdy\\
	&+\iint_{\mathbb{R}^d\times\mathbb{R}^d}w(x-y)|u_r(x)|^2|\chi_{x_1\le0}(y)u_+(y)|^2\,dxdy.
	\end{split}
	\end{equation*}
	The first term is estimated using Young's inequality for $w*|u_r|^2$, then Cauchy-Schwartz in the $y$-integration, and then \eqref{eq:l_r_L2_norm_u+}. The second term is estimated by recognizing that
	\begin{equation*}
	\iint_{\mathbb{R}^d\times\mathbb{R}^d}w(x-y)|u_r(x)|^2|\chi_{x_1\le0}(y)u_+(y)|^2\,dxdy \le C \int_{x_1 \le C}|u_r(x)|^2\,dx\le C_\varepsilon T^{1/2-\varepsilon}
	\end{equation*}
	thanks to \eqref{eq:lr_on_rl}. This shows \eqref{eq:difference_mf_potentials}. Notice that the error estimate in the right hand side of \eqref{eq:difference_mf_potentials} is worse than the one in \eqref{eq:l_r_L2_norm_u+}. This is the reason why we will obtain a similarly worse error estimate in \eqref{eq:l_r_L2_norm_u-}. 
	
	Let us now proceed to the actual proof of \eqref{eq:l_r_L2_norm_u-}. We aim at first proving an upper bound of the form
	\begin{equation} \label{eq:upper_bound_u-}
	\mu_-\le \mu_r+C_\varepsilon T^{1/2-\varepsilon}.
	\end{equation}
	Recall that we have
	\begin{equation*} \label{eq:inf_u-}
	\mu_-=\langle u_-, h_\mathrm{DW} u_-\rangle = \inf\Big\{\langle u,h_\mathrm{DW} u\rangle\;|\;\|u\|_{L^2}=1,\; u\perp u_+\Big\}.
	\end{equation*}
	The function $(u_r-u_\ell)/\|u_r-u_\ell\|_{L^2}$ is then a trial state for this minimization since, by the even parity of $u_+$,
	\begin{equation*}
	\langle u_\ell,u_+\rangle =\langle u_r,u_+\rangle.
	\end{equation*}
	Moreover, using \eqref{eq:scalar_product} we deduce
	\begin{equation*}
	\|u_r-u_\ell\|_{L^2}^2=1-2\mathrm{Re}\langle u_\ell,u_r\rangle \ge 1-C_\varepsilon T^{1-\varepsilon}.
	\end{equation*}
	Hence by the variational principle we have
	\begin{equation} \label{eq:partial_upper_bound_u-}
	\begin{split}
	\mu_-\le\;& \frac{1}{\|u_r-u_\ell\|_{L^2}^2}\langle u_r-u_\ell, h_\mathrm{DW}(u_r-u_\ell)\rangle\\
	\le\;& \langle u_\ell,h_\ell u_\ell\rangle+\langle u_r,h_r u_r\rangle\\
	&-2\int_{\mathbb{R}^d}\nabla u_\ell \nabla u_r\,dx+\int_\ell V_\ell|u_r|^2\,dx+\int_r V_r|u_\ell|^2\,dx\\
	&-2\int_{\mathbb{R}^d} V_\mathrm{DW}u_\ell u_r\,dx-\int_r V_\ell|u_\ell|^2\,dx-\int_\ell V_r|u_r|^2\,dx\\
	&-\lambda\iint_{\mathbb{R}^d\times\mathbb{R}^d}w(x-y)u_\ell(x)u_r(x)|u_+(y)|^2\,dxdy\\
	&+\frac{\lambda}{2}\iint_{\mathbb{R}^d\times\mathbb{R}^d}w(x-y)|u_\ell(x)|^2\big(|u_+(y)|^2-|u_\ell(y)|^2\big)\,dxdy\\
	&+\frac{\lambda}{2}\iint_{\mathbb{R}^d\times\mathbb{R}^d}w(x-y)|u_r(x)|^2\big(|u_+(y)|^2-|u_r(y)|^2\big)\,dxdy\\
	&+C_\varepsilon T^{1-\varepsilon}.
	\end{split}
	\end{equation}
	In the right hand side of \eqref{eq:partial_upper_bound_u-}, the first line equals $2\langle u_r, h_r u_r\rangle$. The second line contains remainders that can be estimated using \eqref{eq:nabla_nabla_term} and \eqref{eq:lr_on_rl_V}. The third and fourth lines are negative and we can safely discard them for an upper bound. The fifth and sixth lines account for the presence of $w*|u_\ell|^2$ and $w*|u_r|^2$ in $h_\ell$ and $h_r$, and their estimate was provided in \eqref{eq:difference_mf_potentials}.
	The only further term in the right hand side of \eqref{eq:partial_upper_bound_u-} is the last line, which comes from the estimate of $\|u_r-u_\ell\|_{L^2}^{-2}$. We therefore proved \eqref{eq:upper_bound_u-}.
	
	The lower bound
	\begin{equation} \label{eq:lower_bound_u-}
	\mu_-=\langle u_-, h_\mathrm{DW} u_-\rangle\ge 2\langle \chi_{x_1\ge0} u_-, h_r \chi_{x_1\ge0} u_-\rangle+C_\varepsilon T^{1/2-\varepsilon}
	\end{equation}
	is easily obtained using IMS formula and proceeding as in the proof of \eqref{eq:lower_bound_u+}, using also \eqref{eq:difference_mf_potentials}, \eqref{eq:localized_without_anything}, and \eqref{eq:localized_gradients}. Comparing \eqref{eq:stability_L2_form}, in which we choose $u=\chi_{x_1\ge0} u_-$ (the symmetry of $|u_-|$ implies that $\|u\|_{L^2}^2=1/2$), with \eqref{eq:upper_bound_u-} and \eqref{eq:lower_bound_u-} we deduce
	\begin{equation*}
	\min_{\theta\in[0,2\pi]}\big\|\chi_{x_1\ge0} u_--e^{i\theta}u_r\big\|_{L^2}^2 \le C_\varepsilon T^{1/2-\varepsilon}.
	\end{equation*}
	Repeating the arguments for the function $\chi_{x_1\le0} u_-$ we also deduce
	\begin{equation*}
	\min_{\theta\in[0,2\pi]}\big\|\chi_{x_1\le0} u_--e^{i\theta}u_\ell\big\|_{L^2}^2 \le C_\varepsilon T^{1/2-\varepsilon}.
	\end{equation*}
	By continuity, the minimizations have to be realized at some $\theta_1$ and $\theta_2$ (that a-priori depend on the distance $L$), i.e.,
	\begin{equation} \label{eq:l_r_L^2_norm_u-}
	\begin{split}
	\big\|\chi_{x_1\ge0} u_--e^{i\theta_1}u_r\big\|_{L^2}^2 \le C_\varepsilon T^{1/2-\varepsilon}\\
	\big\|\chi_{x_1\le0} u_--e^{i\theta_2}u_\ell\big\|_{L^2}^2 \le C_\varepsilon T^{1/2-\varepsilon}.
	\end{split}
	\end{equation}
	Moreover, since $u_-$ and $u_r$ are real-valued functions, we have
	\begin{equation*}
	\big\|\chi_{x_1\ge0} u_--e^{i\theta_1}u_r\big\|_{L^2}^2=1-\cos\theta_1 \int_{\mathbb{R}^d}\chi_{x_1\ge0} u_- u_r\,dx,
	\end{equation*}
	and similarly for the other norm. This shows that the minimization can occur for $\theta_1\in\{0,\pi\}$ and $\theta_2\in\{0,\pi\}$ only, depending on the sign of the integral in the right hand side. Modulo a change of sign of $u_-$, we can certainly assume that $\theta_1=0$ for every $L$. There remains to show that $\theta_2=\pi$ for every $L$. But this follows from
	\begin{equation*}
	\begin{split}
	0=\;&\langle u_-,u_+\rangle= \left\langle \chi_{x_1\le0} u_-,\chi_{x_1\le0}u_+\right\rangle+\left\langle \chi_{x_1\ge0} u_-,\chi_{x_1\ge0} u_+\right\rangle \\
	=\;&\left\langle \left(\chi_{x_1\le0} u_--e^{i\theta_2}u_\ell\right), \chi_{x_1\le0}u_+\right\rangle+\left\langle \left(\chi_{x_1\ge0} u_--u_r\right), \chi_{x_1\ge0} u_+\right\rangle\\
	&+ e^{-i\theta_2} \left\langle u_\ell, \chi_{x_1\le0} u_+\right\rangle+ \left\langle u_r,\chi_{x_1\ge0} u_+\right\rangle.
	\end{split}
	\end{equation*}
	Indeed, the first two scalar products in the right hand side converge to zero due to \eqref{eq:l_r_L^2_norm_u-} (with the choice $\theta_1=0$), while the second two scalar products converge to one due to \eqref{eq:l_r_L2_norm_u+}. This shows that $\theta_1$ must converge to $\pi$ as $L$ tends to infinity. Since it can only attain the values $0$ and $\pi$ by what discussed above, it must coincide with $\pi$ for $L$ large enough.
\end{proof}

\section{Estimates on spectral gaps}\label{sec:gaps}

In the present section we prove the claims \eqref{eq:first_gap} and \eqref{eq:second_gap} from our main result. The proof of the lower bound in \eqref{eq:first_gap}, being the most involved, requires an extra amount of information on $u_+$ and $u_-$, beyond the preliminary estimates of Proposition~\ref{prop:decay_estimates}. We discuss this in Subsection~\ref{subsect:further_u_+}.

\subsection{Upper bound on the first gap}

To deduce the upper bound in \eqref{eq:first_gap} we  consider the function
\begin{equation*}
v:=\frac{\big(\chi_{x_1\ge0}-\chi_{x_1\le0}\big) u_+}{\big\|\big(\chi_{x_1\ge0}-\chi_{x_1\le0}\big) u_+\big\|_{L^2}} \perp u_+
\end{equation*}
as a trial function for the minimization
\begin{equation*}
\begin{split}
\mu_-=\;&\min\Big\{\langle u,h_{\mathrm{DW}}u\rangle,\;\|u\|_{L^2}=1,\;u \perp u_+\Big\}.
\end{split}
\end{equation*}
Here $\chi_{x_1\ge0}$ and $\chi_{x_1\le0}$ are localization functions as in Proposition \ref{prop:first_convergence}, and this ensures
\begin{equation} \label{eq:odd_function_norm}
\big\|\big(\chi_{x_1\ge0}-\chi_{x_1\le0}\big) u_+\big\|_{L^2}^2=1-2\int_{\mathbb{R}^d}\chi_{x_1\ge0}\chi_{x_1\le0} |u_+|^2\,dx \ge 1-C_\varepsilon T^{1-\varepsilon}
\end{equation}
thanks to \eqref{eq:localized_without_anything}. By the variational principle we have
\begin{equation*}
\mu_-\le \langle v, h_\mathrm{DW} v\rangle = \frac{2\langle \chi_{x_1\ge0} u_+,h_\mathrm{DW}\chi_{x_1\ge0} u_+\rangle - 2 \langle \chi_{x_1\ge0} u_+,h_\mathrm{DW} \chi_{x_1\le0}u_+\rangle}{\big\|\big(\chi_{x_1\ge0}-\chi_{x_1\le0}\big) u_+\big\|_{L^2}^2}.
\end{equation*}
The second term in the numerator is bounded in absolute value by $C_\varepsilon T^{1-\varepsilon}$ as can be seen using \eqref{eq:localized_gradients}, \eqref{eq:chichi_V_+}, and \eqref{eq:chichi_w_+}. For the first term we use IMS formula which implies
\begin{equation} \label{eq:lower_bound_u_+}
\langle u_+, h_\mathrm{DW} u_+\rangle\ge  2\langle \chi_{x_1\ge0} u_+, h_\mathrm{DW} \chi_{x_1\ge0} u_+\rangle-C_\varepsilon T^{1-\varepsilon},
\end{equation}
as can be seen using once again \eqref{eq:localized_gradients}, \eqref{eq:chichi_V_+}, and \eqref{eq:chichi_w_+}. This, together with \eqref{eq:odd_function_norm} gives
\begin{equation*}
\mu_-\le \langle u_+,h_\mathrm{DW} u_+\rangle + C_\varepsilon T^{1-\varepsilon}= \mu_++CT^{1-\varepsilon}.
\end{equation*}

\subsection{Further properties of $u_+$ and $u_-$} \label{subsect:further_u_+}

We start with the following proposition, which provides a lower bound analog to~\eqref{eq:pointwise_upper_bound} for $u_+$, vindicating the sharpness of the latter estimate.

\begin{proposition}[\textbf{Lower bound for $u_+$}]\mbox{}\label{prop:lower_u_+}\\ 
	Let
	\begin{equation*}
	\alpha:=\begin{cases}
	\frac{2d-2+s}{4s}\quad&s>2\\
	\frac{2d-2+s}{4s}-\frac{\mu_{+}}{2s}\quad&s=2.
	\end{cases}
	\end{equation*}
	Then there exists $R>0$ independent of $L$ such that, for every $0<\varepsilon<1$, there exists a constant $c_\varepsilon>0$, also independent of $L$, such that
	\begin{equation} \label{eq:lower_u_+}
	u_+(x)\ge c_\varepsilon \frac{e^{-A_\mathrm{DW}(x)}}{V_\mathrm{DW}(x)^{\alpha_++\varepsilon}}
	\end{equation}
	for any $x$ such that $|x-\mathbf{x}|,|x+\mathbf{x}|\ge R$.
\end{proposition}

A pointwise lower bound for $u_-$ and $u_\mathrm{ex}$ is not to be expected, since excited eigenfunctions have to change sign.

\begin{proof}[Proof]For a number $\beta\in\mathbb{R}$, define the function
	\begin{equation*}
	f(x)=e^{-A_\mathrm{DW}(x)}V_\mathrm{DW}(x)^{-\beta/s}=\begin{cases} 
	e^{-(1+s/2)^{-1}|x-\bx|^{1+s/2}}|x-\bx|^{-\beta}\quad x_1\ge0\\
	e^{-(1+s/2)^{-1}|x+\bx|^{1+s/2}}|x+\bx|^{-\beta}\quad x_1\le0.
	\end{cases}
	\end{equation*}
	Using the fact that $f$ only depends on $|x-\bx|$ for $x_1 \geq 0$ or $|x+\bx|$ for $x_1\leq 0$, we compute
	\begin{equation}\label{eq:trial sol}
	\begin{split}
	\Delta& f(x)\\
	&=\begin{cases} 
	\Big[|x-\bx|^s+\Big(2\beta-\frac{s}{2}-d+1\Big)|x-\bx|^{s/2-1}+\Big(\beta^2+2\beta-d\beta\Big)|x-\bx|^{-2}\Big] f(x)\\\text{for} \quad x_1>0\\\\
	\Big[|x+\bx|^s+\Big(2\beta-\frac{s}{2}-d+1\Big)|x+\bx|^{s/2-1}+\Big(\beta^2+2\beta-d\beta\Big)|x+\bx|^{-2}\Big] f(x)\\\text{for} \quad x_1>0.
	\end{cases}
	\end{split}
	\end{equation}
	Notice that for $x_1=0$ the Laplacian is defined in a distributional sense only.
	Since $w*|u_+|^2-\mu_+$ is uniformly $L^\infty$-bounded (by Young's inequality and Lemma \ref{lemma:uniform_mu}), picking $\beta=s\alpha_++\varepsilon$ we deduce
	\begin{equation} \label{eq:eventual_negativity}
	\Big(-\Delta+V_\mathrm{DW}(x)+\lambda w*|u_+|^2(x)-\mu_+\Big)f(x)\le 0
	\end{equation}
	for $|x-\bx|\ge R$ and $|x+\bx|\ge R$ with $R$ large enough and independent on $L$, and away from $x_1=0$. Notice that, in the case $s=2$, the existence of such an $R$ is ensured by the fact that $w*|u_+|^2$ decays at infinity with a rate that does not depend on $L$, as follows from the upper bound in~\eqref{eq:pointwise_upper_bound}.	
	Moreover, again for $|x-\bx|\ge R$ and $|x+\bx|\ge R$ with $R$ large enough and independent on $L$,
	\begin{equation} \label{eq:eventual_control}
	V_\mathrm{DW}(x)>\mu_+-\lambda w*|u_+|^2(x).
	\end{equation}
	Consider now a function $f_\mathrm{low}$ which is equal to $ f$ (with $\beta=s\alpha_++\varepsilon$) outside of $B_{\bx}(R)\cup B_{-\bx}(R)$ and smoothly extended to a function bounded away from zero inside $B_{\bx}(R)\cup B_{-\bx}(R)$. Since, by construction, $f$ is a $L$-independent function of $|x-\mathbf{x}|$ or $|x+\mathbf{x}|$, $f_\mathrm{low}$ can be chosen to be $L$-independent as well. Define the (a priori $L$-dependent) constant
	\begin{equation*}
	c_\varepsilon=\min_{\substack{|x-\bx|<R\\|x+\bx|<R}}\frac{u_+(x)}{f_\mathrm{low}(x)}.
	\end{equation*}
	Since $u_+>0$, we have\footnote{This is the place where the argument ceases to apply to $u_-$ or $u_{\rm ex}$} that $c_\varepsilon>0$. Let us consider the continuous function
	\begin{equation*}
	g=u_+-c_\varepsilon f_\mathrm{low}.
	\end{equation*}
	We will prove that $g$ is positive. Then \eqref{eq:lower_u_+} follows thanks to
	\begin{equation*}
	c_\varepsilon \ge \frac{\min_{\substack{|x-\bx|<R\\|x+\bx|<R}} u_+(x)}{\max_{\substack{|x-\bx|<R\\|x+\bx|<R}}f_\mathrm{low}} \ge C \|u_+\|_{L^\infty\left(B_{\bx}(R)\cup B_{-\bx}(R)\right)}.
	\end{equation*}
	Here we used Harnack's inequality~\cite[Section~6.4.3]{Evans-98} for the set $B_{\bx}(R)\cup B_{-\bx}(R)$ (the resulting constant depends only on $R$, not on $L$). The right-hand side of the above is bounded away from $0$ as $L\to\infty$, for otherwise $u_+$ would converge to $0$ in $L^2 \left(B_{\bx}(R)\cup B_{-\bx}(R)\right)$, contradicting Proposition~\ref{prop:first_convergence}. Thus $c_\varepsilon$ is bounded away from zero as $L\to\infty$ and there only remains to prove that $g$ is positive.
	
	The function $g$, being continuous and decaying at infinity, could have either of the three following behaviors. (i) it attains a global minimum inside $B_\mathbf{x}(R)\cup B_{-\mathbf{x}}(R)$, (ii) it attains a global minimum outside of $B_\mathbf{x}(R)\cup B_{-\mathbf{x}}(R)$, or  (iii) it has no global minimum and is everywhere positive. In the latter case (iii) there is nothing to prove. In case (i) the proof is complete because $g$ is by construction positive inside $B_\mathbf{x}(R)\cup B_{-\mathbf{x}}(R)$. Finally, consider case (ii): $g$ attains a global minimum at some $(y_1,\ldots,y_d) = y\notin B_\mathbf{x}(R)\cup B_{-\mathbf{x}}(R)$. We can exclude the possibility that $y_1=0$. Indeed, on the hyperplane $\{x_1=0\}$, $u_+$ has (by parity) a $x_1$-directional critical point where (by smoothness) $\nabla u_+$ vanishes. On the other hand, $f_\mathrm{low}$ has by construction a $x_1$-directional minimum there, with $\partial_{x_1} f_\mathrm{low}$ making a jump between two different \emph{non-zero} values. Hence, in the $x_1$ direction, $u_+$ converges at least quadratically to its value at $x_1=0$, while $f_\mathrm{low}$ decays as $\sim|x_1|$ to its value on the hyperplane. It follows that $g$ has a $x_1$-directional maximum on the hyperplane $x_1 = 0$. Thus the global minimum $y$ must belong to a region where $g$ is smooth, implying $\Delta g \geq 0$ at $y$. We then use the eigenvalue equation for $u_+$ and \eqref{eq:eventual_negativity} to deduce
	\begin{equation*}
	\begin{split}
	\big(-\Delta +V_\mathrm{DW}(y)&+\lambda w*|u_+|^2(y)-\mu_+\big) g(y) \\
	=\;&-c_\varepsilon\left(-\Delta +V_\mathrm{DW}(y)+\lambda w*|u_+|^2(y)-\mu_+\right)f_\mathrm{low} (y)\ge0
	\end{split}
	\end{equation*}
	and hence
	\begin{equation*}
	\big( V_\mathrm{DW}(y)+\lambda w*|u_+|^2(y)-\mu_+\big) g(y) \ge0.
	\end{equation*}
	Thanks to \eqref{eq:eventual_control} we finally deduce $g(y)\ge0$, which concludes the proof.	
\end{proof}

We also state the following lemma, containing properties of $u_-$.

\begin{lemma}[\textbf{Symmetry and sign of $u_-$}] \label{lemma:positivity_u-}\mbox{}\\
	The function $u_-$ is odd with respect to reflections across the $x_1=0$ plane, i.e.,
	\begin{equation} \label{eq:odd_symmetry}
	u_-(-x_1,x_2,\dots,x_d)=-u_-(x_1,x_2,\dots,x_d).
	\end{equation}
	Moreover, assume that we pick for $u_-$ the same phase as in Proposition \ref{prop:first_convergence}. Then $u_-(x)>0$ almost everywhere for $x_1>0$.
\end{lemma}
%
%

\begin{proof}[Proof of Lemma \ref{lemma:positivity_u-}]
	Since $u_-$ must be either odd or even, the fact that it is odd is a consequence of \eqref{eq:l_r_L2_norm_u-}. To prove that $u_-(x)>0$ for $x_1>0$, let us first notice that by the odd symmetry
	\begin{equation*}
	\begin{split}
	\mu_-=\;&\min\Big\{\langle u,h_{\mathrm{DW}}u\rangle,\;\|u\|_{L^2}=1,\;u \perp u_+\Big\}\\
	=\;&\min\Big\{\langle u,h_{\mathrm{DW}}u\rangle,\;\|u\|_{L^2}=1,\;u\;\text{odd}\Big\}.
	\end{split}
	\end{equation*}
	Thus, $\mu_-$ must coincide with the ground state energy of the Dirichlet Hamiltonian
	\begin{equation*}
	H=-\Delta^{(D)}_{x_1\ge0}+V_\mathrm{DW}+\lambda w*|u_+|^2,
	\end{equation*}
	where $\Delta^{(D)}_{x_1\ge0}$ is the Dirichlet Laplacian in the half-space $\{x_1\ge0\}$. Using the Trotter product formula~\cite[Theorem VII.31]{ReeSim1} and the fact that $V_\mathrm{DW}+\lambda w*|u_+|^2\ge 0$, it is easy to see that $e^{-tH}$ is positivity improving $\forall t>0$. On the other hand $H$ has compact resolvent, and hence the bottom of its spectrum is an eigenvalue. Using~\cite[Theorem XIII.44]{ReeSim4} completes the proof. 
\end{proof}

\subsection{Lower bound for the first gap}

We now provide the proof of the lower bound in \eqref{eq:first_gap}, following~\cite{Simon-84}. We start with

\begin{lemma}[\textbf{Expression for the first gap}]\mbox{}\\
	For any smooth $f$ such that $fu_+\in H^1(\mathbb{R}^d)$, we have
	\begin{equation}
	\langle fu_+,(h_\mathrm{MF}-\mu_+) f u_+\rangle =\frac{1}{2}\left\| u_+\nabla f\right\|^2_{L^2}.
	\end{equation}
\end{lemma}

\begin{proof}
	The proof (from~\cite{KacTho-69,Simon-84}) follows from the equality
	\begin{equation*}
	\left[ f,\left[f,h_\mathrm{MF}-\mu_+  \right] \right]=\left[ f,\left[ f,-\frac{1}{2}\Delta \right] \right]=-\left( \nabla f\right)^2
	\end{equation*}
	after taking the expectation on $u_+$.
\end{proof}

In order to prove the lower bound, we apply the above with $f=u_-/u_+$, and drop the integration in the right-hand side outside of a hyper-rectangle. Notice that $f$ is smooth since both $u_+$ and $u_-$ are, and $u_+$ is strictly positive (see, e.g., \cite[Thm. XIII.47]{ReeSim4}). This gives
\begin{equation*}
\mu_--\mu_+ =\frac{1}{2} \int_{\mathbb{R}^d} |\nabla f|^2 |u_+|^2 \ge \frac{1}{2}\int_{x_2=-K}^{x_2=K}\dots \int_{x_d=-K}^{x_d=K}\int_{x_1=-L/2+R}^{x_1=L/2-R} |\nabla f|^2 |u_+|^2 \,dx_1\dots dx_d,
\end{equation*}
where $K$ is an arbitrary fixed positive number, and $R$ is large enough so that one can apply Proposition \ref{prop:lower_u_+}, which we do. For a lower bound on $u_+$ we use~\eqref{eq:lower_u_+} and evaluate $A_\mathrm{DW}$ at its minimum $x=0$ in the right-hand side. This gives
\begin{equation*}
\mu_--\mu_+ \ge C_\varepsilon T^{1+\varepsilon}\int_{x_2=-K}^{x_2=K}\dots \int_{x_d=-K}^{x_d=K}\int_{x_1=-L/2+R}^{x_1=L/2-R} |\nabla f|^2\,  dx_1\dots dx_d.
\end{equation*}
For the gradient of $f$ we use the trivial inequality
\begin{equation*}
|\nabla f|^2 \ge |\partial_1 f|^2,
\end{equation*}
as well as Cauchy-Schwarz in the form
\begin{equation*}
\begin{split}
\int_{x_2=-K}^{x_2=K}\dots& \int_{x_d=-K}^{x_d=K}\int_{x_1=-L/2+R}^{x_1=L/2-R} |\partial_1 f|^2 \, dx_1\dots dx_d \\
\ge\;& \frac{1}{C L}\left( \int_{x_2=-K}^{x_2=K}\dots \int_{x_d=-K}^{x_d=K}\int_{x_1=-L/2+R}^{x_1=L/2-R} \partial_1 f\,  dx_1\dots dx_d\right)^2\\
=\;&\frac{1}{C L}\bigg( \int_{x_2=-K}^{x_2=K}\dots \int_{x_d=-K}^{x_d=K} \\
&\qquad\times\left[ f(L/2-R,x_2,\dots,x_d)-f(-L/2+R,x_2,\dots,x_d) \right]  \,dx_2\dots dx_d\bigg)^2\\
=\;&\frac{2}{CL}\bigg( \int_{x_2=-K}^{x_2=K}\dots \int_{x_d=-K}^{x_d=K}  f(L/2-R,x_2,\dots,x_d) \, dx_2\dots dx_d\bigg)^2.
\end{split}
\end{equation*}
The last step follows from the fact that $f=u_-/u_+$ is odd under reflection around the $x_1=0$ axis, since $u_-$ is odd (see Lemma \ref{lemma:positivity_u-}) and $u_+$ is even. 
Since the $L^{-1}$ factor can be absorbed inside $T^{1+\varepsilon}$ with a slight modification of $\varepsilon$, we only have to show that the integral in the right hand side is bounded away from zero uniformly in $L$. This is a consequence of \eqref{eq:strong CV} and of the Sobolev embedding (see the discussion after Proposition~\ref{prop:first_convergence}), which imply that $u_+$ and $u_-$ converge in $L^\infty$ to the same function in $ B(\bx,R)$ for $R$ fixed. This completes the proof.

\subsection{Lower bound on the second gap}\label{sec:second gap}

To prove \eqref{eq:second_gap}, recall that $u_\mathrm{ex}$ is the first excited state above $u_-$, i.e., $\mu_{ex}=\langle u_\mathrm{ex},h_\mathrm{DW} u_\mathrm{ex}\rangle$. A lower bound for $\mu_\mathrm{ex}$ follows from the IMS formula and reads
\begin{equation*}
\mu_\mathrm{ex}\ge 2\langle\chi_{x_1\ge0} u_\mathrm{ex},h_\mathrm{DW}\chi_{x_1\ge0} u_\mathrm{ex}\rangle-C_\varepsilon T^{1-\varepsilon},
\end{equation*}
having used \eqref{eq:localized_gradients_ex} and \eqref{eq:localized_without_anything_ex}. Here $\chi_{x_1\ge0}$ and $\chi_{x_1\le0}$ are localization functions as in Proposition \ref{prop:first_convergence}. We further argue that \eqref{eq:V_DW_to_V_r_ex} and \eqref{eq:difference_mf_potentials} allow to replace $h_\mathrm{DW}$ with $h_r$, i.e.,
\begin{equation} \label{eq:partial_lower_bound_large_gap}
\mu_\mathrm{ex}\ge 2\langle\chi_{x_1\ge0} u_\mathrm{ex},h_r\chi_{x_1\ge0} u_\mathrm{ex}\rangle-C_\varepsilon T^{1/2-\varepsilon}.
\end{equation}
We now want to bound from below the right hand side using (a suitable modification of) $\chi_{x_1\ge0}u_\mathrm{ex}$ as a trial state for the minimization problem
\begin{equation*}
\mu_{r,\mathrm{ex}}:=\inf\Big\{ \langle u, h_r  u\rangle\;|\; \|u\|_{L^2}=1,\; u\perp u_r \Big\}.
\end{equation*}
Define then
\begin{equation*}
v:= \frac{ \chi_{x_1\ge0} u_\mathrm{ex}-2 \langle u_r,\chi_{x_1\ge0}u_\mathrm{ex}\rangle u_r}{\big\|\chi_{x_1\ge0} u_\mathrm{ex}-2 \langle u_r,\chi_{x_1\ge0}u_\mathrm{ex}\rangle u_r\big\|_{L^2}}.
\end{equation*}
By construction $v$ is orthogonal to $u_r$ (since $\|u_r\|_{L^2}^2=1/2$), which makes it a trial function for the $\mu_{r,\mathrm{ex}}$ minimization. We want to estimate the norm in the denominator. Recall that $u_\mathrm{ex}$ must be either even or odd under reflection across the $\{x_1=0\}$ hyperplane. Assume it is even. Then
\begin{equation*}
0=\langle u_+, u_\mathrm{ex}\rangle= \langle \chi_{x_1\ge0}u_+, \chi_{x_1\ge0}u_\mathrm{ex}\rangle+\langle\chi_{x_1\le0} u_+, \chi_{x_1\le0}u_\mathrm{ex}\rangle=2\langle \chi_{x_1\ge0}u_+, \chi_{x_1\ge0}u_\mathrm{ex}\rangle,
\end{equation*}
which implies
\begin{equation*}
\begin{split}
\big\|\chi_{x_1\ge0} u_\mathrm{ex}-2 \langle u_r,\chi_{x_1\ge0}u_\mathrm{ex}\rangle u_r\big\|_{L^2}^2=\;&\frac{1}{2}-2 \big| \langle u_r,\chi_{x_1\ge0}u_\mathrm{ex}\rangle \big|^2\\
=\;&\frac{1}{2}-2\big| \big\langle \big( u_r-\chi_{x_1\ge0}u_+ \big),\chi_{x_1\ge0} u_\mathrm{ex}\big\rangle \big|^2\\
\ge\;&\frac{1}{2}-C_\varepsilon T^{1-\varepsilon},
\end{split}
\end{equation*}
where the inequality follows by Cauchy-Schwartz and by approximation of $u_+$ that we deduced in \eqref{eq:l_r_L2_norm_u+}. If, on the other hand, $u_\mathrm{ex}$ is odd, then we can repeat the same calculation with $u_+$ replaced by $u_-$. The variational principle then implies
\begin{equation*}
\begin{split}
\mu_{r,\mathrm{ex}} \le\;& \frac{\langle \chi_{x_1\ge0} u_\mathrm{ex},h_r \chi_{x_1\ge0} u_\mathrm{ex}\rangle-2\mu_r \big| \langle u_r,\chi_{x_1\ge0}u_\mathrm{ex}\rangle \big|^2}{\big\|\chi_{x_1\ge0} u_\mathrm{ex}-2 \langle u_r,\chi_{x_1\ge0}u_\mathrm{ex}\rangle u_r\big\|_{L^2}^2}\\
\le\;&2\langle \chi_{x_1\ge0} u_\mathrm{ex},h_r \chi_{x_1\ge0} u_\mathrm{ex}\rangle+C_\varepsilon T^{1-\varepsilon},
\end{split}
\end{equation*}
having ignored the second term in the numerator because it is negative. Comparing this with \eqref{eq:partial_lower_bound_large_gap} we find
\begin{equation*}
\mu_{\mathrm{ex}}\ge \mu_{r,\mathrm{ex}}-C_\varepsilon T^{1/2-\varepsilon}.
\end{equation*}
Now, we know that the spectrum of $h_r$ does not depend on $L$, since $h_r$ coincides with the translation of a fixed Hamiltonian. Hence, the gap between $\mu_r$ and $\mu_{r,\mathrm{ex}}$ is a fixed constant. Moreover, by \eqref{eq:upper_bound_u-}, we have
\begin{equation*}
\mu_r \ge  \mu_--C_\varepsilon T^{1/2-\varepsilon}.
\end{equation*}
This gives
\begin{equation*}
\mu_{\mathrm{ex}}\ge \mu_{r,\mathrm{ex}}-C_\varepsilon T^{1/2-\varepsilon} \ge \mu_r +C-C_\varepsilon T^{1/2-\varepsilon}  \ge  \mu_- +C-C_\varepsilon T^{1/2-\varepsilon},
\end{equation*}
which proves \eqref{eq:second_gap}.

\section{Refined estimates, following Helffer-Sj\"ostrand} \label{sect:refined}

The aim of this section is to prove~\eqref{eq:L1_convergence},~\eqref{eq:L2_convergence}, and~\eqref{eq:Linfty_convergence}. The optimal rate in \eqref{eq:L1_convergence} will follow from a careful choice of approximating quasi-modes inspired by~\cite{HelSjo-84,HelSjo-85}. 

Let us denote by $u^{(D)}_r$ the (normalized) ground state of the Dirichlet problem
\begin{equation} \label{eq:dirichlet_r}
\begin{cases}
\Big(-\Delta+V_\mathrm{DW}+\lambda w*|u_+|^2\Big)u=\mu u\\
u(x)=0,\qquad\text{for } x_1\le -\frac{L}{2}+c
\end{cases}
\end{equation}
with eigenvalue $\mu^{(D)}$. Let us, in turn, denote by $u_\ell^{(D)}$ the (normalized) ground state of the Dirichlet problem
\begin{equation}  \label{eq:dirichlet_ell}
\begin{cases}
\Big(-\Delta+V_\mathrm{DW}+\lambda w*|u_+|^2\Big)u=\mu u\\
u(x)=0,\qquad\text{for } x_1\ge -\frac{L}{2}+c.
\end{cases}
\end{equation}
By symmetry across the $\{x_1=0\}$ hyperplane we have $u_\ell^{(D)}(-x_1,x_2,\dots,x_d)=u_r^{(D)}(x)$ and therefore the eigenvalue corresponding to $u^{(D)}_\ell$ coincides with $\mu^{(D)}$.

The cutoff distance $c>0$ will eventually be chosen to depend (non-uniformly) on the parameter $\varepsilon$ appearing in the right hand side of \eqref{eq:L1_convergence}, which we will take arbitrarily small. As a consequence, since most quantities depends on $c$, they will implicitly depend on $\varepsilon$. We will however not keep track of such a dependence.

\subsection{Agmon decay estimates}

The first step in the proof of \eqref{eq:L1_convergence} is to show suitable decay estimates for $u^{(D)}_r$ and $u^{(D)}_\ell$. These will be more refined that what we have proved so far. 

We start by defining the double-well Agmon distance between two points $x,y\in\mathbb{R}^d$
\begin{equation} \label{eq:Agmon_distance}
\begin{split}
d_\mathrm{DW}(x,y)=\;&\inf_{\gamma\text{ piecewise }C^1\text{ curve}}\Big\{ \int_0^1\sqrt{V_\mathrm{DW}(\gamma(t))}|\gamma'(t)|\,dt,\;|\;\gamma(0)=x,\,\gamma(1)=y  \Big\}.
\end{split}
\end{equation}
The exponentials of the functions $d_\mathrm{DW}(\,\cdot\,,\bx)$ and $d_\mathrm{DW}(\,\cdot\,,-\bx)$ will model the decay of, respectively, $u^{(D)}_r$ and $u^{(D)}_\ell$. The following general properties are well-known (see, e.g., \cite[Equations (3.2.1) and (3.2.2)]{Helffer-88})
\begin{align} \label{eq:first_property_d}
d_\mathrm{DW}(x,y)\le&\; d_\mathrm{DW}(x,z)+d_\mathrm{DW}(z,y),\quad\forall x,y,z \qquad\text{(triangular inequality)}\\\label{eq:second_property_d}
|\nabla_x d_\mathrm{DW}(x,y)|^2\le&\; V_\mathrm{DW}(x),\quad\forall x,y.
\end{align}

Furthermore, we have the following Lemma.

\begin{lemma}[\textbf{Properties of the double-well Agmon distance}]\mbox{}\label{lemma:distance}\\ 
	The function $d_\mathrm{DW}(\,\cdot\,,\bx)$ satisfies the three following properties, with $A$ the single-well Agmon distance~\eqref{eq:Agmon} and $c$ the constant in~\eqref{eq:dirichlet_r}-~\eqref{eq:dirichlet_ell}:
	\begin{itemize}
		\item[$(i)$]\textbf{First estimate in the half-space:}
		\begin{equation} \label{eq:distance_first_estimate}
		\begin{split}
		d_{\mathrm{DW}}(x,\bx)\ge\;& A(|x-\bx|)\qquad x_1\ge0,\\
		d_{\mathrm{DW}}(x,-\bx)\ge\;& A(|x+\bx|)\qquad x_1\le0.
		\end{split}
		\end{equation}
		\item[$(ii)$] \textbf{Estimate at }$\mathbf{(x_2,\dots,x_d)=0}$\textbf{:}
		\begin{equation}\label{eq:distance_line}
		\begin{split}
		d_\mathrm{DW}\big((x_1,0,\dots,0),\bx\big)\ge\;&\begin{cases}
		A\big(\big|x_1-\frac{L}{2}\big|\big),\quad &x_1\ge0\\
		2A(\frac{L}{2})-A\big( \big| \frac{L}{2}+x_1 \big| \big),\quad &-L/2+c \le x_1\le0.
		\end{cases}\\
		d_\mathrm{DW}\big((x_1,0,\dots,0),-\bx\big)\ge\;&\begin{cases}
		A\big(\big|x_1+\frac{L}{2}\big|\big),\quad &x_1\le0\\
		2A(\frac{L}{2})-A\big( \big| \frac{L}{2}-x_1 \big| \big),\quad &0 \le x_1\le \frac{L}{2}-c.
		\end{cases}
		\end{split}
		\end{equation}
		\item[$(iii)$]\textbf{Second estimate in the half space:}
		\begin{equation} \label{eq:distance_second_estimate}
		\begin{split}
		d_\mathrm{DW}\big( (x_1,x_2,\dots,x_d),\bx \big)\ge\;&2A\Big(\frac{L}{2}\Big)-A\Big( \Big| \frac{L}{2}+x_1 \Big| \Big),\qquad -\frac{L}{2}+c \le x_1\le0.\\
		d_\mathrm{DW}\big( (x_1,x_2,\dots,x_d),-\bx \big)\ge\;&2A\Big(\frac{L}{2}\Big)-A\Big( \Big| \frac{L}{2}-x_1 \Big| \Big),\qquad 0 \le x_1\le \frac{L}{2}-c.
		\end{split}
		\end{equation}
	\end{itemize}
\end{lemma}

\begin{proof}
	For each of the three points we will only prove the property for $d_\mathrm{DW}(\,\cdot\,,\bx)$, since the one for $d_\mathrm{DW}(\,\cdot\,,-\bx)$ can be then deduced by reflection symmetry.
	
	Let us prove $(i)$. First, notice that, in the $x_1\ge0$ region, $V_\mathrm{DW}$ only depends on the radial coordinate $|x-\bx|$. Hence, the same must be true for $d_\mathrm{DW}(\,\cdot\,,\bx)$, and thus, without loss of generality, we can assume $x=(x_1,0,\dots,0)$, the general case being then deduced from
	\begin{equation*}
	d_\mathrm{DW}(x,\bx)=d_\mathrm{DW}\Big(\big(\big||x-\bx|-L/2\big|,0,\dots,0\big),\bx\Big).
	\end{equation*}
	Let us now prove that, in order to compute $d_\mathrm{DW}\big((x_1,0,\dots,0),\bx\big)$ for $x_1\ge0$ we can reduce ourselves, in the definition of $d_\mathrm{DW}(\,\cdot\,,\bx)$, to curves supported on the line $x_2,\dots,x_d=0$ only. Indeed, for any piecewise $C^1$ curve $\gamma:[0,1]\to\mathbb{R}^d$ such that $\gamma(0)=\bx$ and $\gamma(1)=(x_1,0,\dots,0)$, let us define the curve projected onto the $x_2,\dots,x_d=0$ line
	\begin{equation*}
	\gamma_1(t):=(\gamma(t)_1,0,\dots,0).
	\end{equation*}
	Then, by definition,
	\begin{equation*}
	|\gamma_1'(t)|\le|\gamma'(t)|.
	\end{equation*}
	Since $V_\mathrm{DW}(y)\ge V_\mathrm{DW}\big( (y_1,0,\dots,0) \big)$ for any $y\in\mathbb{R}^d$, we find	
	\begin{equation*}
	\int_0^1\sqrt{V_\mathrm{DW}(\gamma(t))}|\gamma'(t)|\,dt \ge \int_0^1\sqrt{V_\mathrm{DW}(\gamma_1(t))}|\gamma_1'(t)|\,dt.
	\end{equation*}
	This shows that it is always favorable to only consider paths restricted to the line. Let then $\widetilde{\gamma}:[0,1]\to\mathbb{R}$ be a piecewise $C^1$ curve such that $\widetilde{\gamma}(0)=L/2$ and $\widetilde{\gamma}(1)=x_1$. We have
	\begin{equation*}
	\begin{split}
	A(|x-\bx|)=\;&\frac{1}{1+\frac{s}{2}}\Big|x_1-\frac{L}{2}\Big|^{1+\frac{s}{2}}\\
	=\;&\frac{1}{1+\frac{s}{2}}\int_0^1\frac{d}{dt}\Big|\widetilde\gamma(t)-\frac{L}{2}\Big|^{1+\frac{s}{2}}dt\\
	\le\;&\int_0^1\Big|\widetilde{\gamma}(t)-\frac{L}{2}\Big|^{s/2}\,|\widetilde{\gamma}'(t)|\,dt\\
	=\;&\int_0^1 \sqrt{V_\mathrm{DW}\big((\widetilde{\gamma}(t),0,\dots,0)\big)}|\widetilde\gamma'(t)|\,dt.
	\end{split}
	\end{equation*}
	Considering the infimum over all such curves $\widetilde{\gamma}$ (which we proved above to coincide with the infimum over all curves), we deduce \eqref{eq:distance_first_estimate}.
	
	Let us now prove $(ii)$. The claim for $x_1\ge0$ already follows from \eqref{eq:distance_first_estimate}. We concentrate on $x_1\le0$. By repeating the arguments used above, one easily sees that in order to compute $d_\mathrm{DW}((x_1,0,\dots,0),\bx)$ for $x_1\le0$ it is again convenient to restrict to curves supported on $x_2,\dots,x_d=0$. Let then $\widetilde{\gamma}:[0,1]\to\mathbb{R}$ be any piecewise $C^1$ curve such that $\widetilde{\gamma}(0)=L/2$ and $\widetilde{\gamma}(1)=x_1$. Since $x_1\le0$, there exists a time $t_{\widetilde{\gamma}}$, depending on the choice of the curve, such that $\widetilde{\gamma}(t_{\widetilde\gamma})=0$. We then write
	\begin{equation*}
	\begin{split}
	2A\Big(\frac{L}{2}\Big)-A\Big(\, \Big|\frac{L}{2}+x_1\Big|\, \Big)=\;&\int_0^{t_{\widetilde{\gamma}}}\frac{d}{dt}A\Big(\, \Big| \widetilde\gamma(t)-\frac{L}{2} \Big|\,\Big)\,dt\\
	&+\int_{t_{\widetilde{\gamma}}}^1\frac{d}{dt}\Big[ 2A\Big(\frac{L}{2}\Big)-A\Big(\,\Big|\frac{L}{2}+\widetilde{\gamma}(t)\Big|\,\Big) \Big]\,dt\\
	\le\;& \int_0^{t_{\widetilde{\gamma}}}\Big| \widetilde{\gamma}(t)-\frac{L}{2} \Big|^{s/2}|\widetilde{\gamma}'(t)|\,dt+\int_{t_{\widetilde{\gamma}}}^1\Big| \widetilde{\gamma}(t)+\frac{L}{2} \Big|^{s/2}|\widetilde{\gamma}'(t)|\,dt\\
	=\;&\int_0^1 \sqrt{V_\mathrm{DW}\big(\widetilde{\gamma}(t)\big)}|\widetilde{\gamma}'(t)|\,dt.
	\end{split}
	\end{equation*}
	Taking the infimum over all such curves $\widetilde{\gamma}$ yields the result.
	
	Finally, $(iii)$ is deduced by arguing, as done above, that projecting a curve onto the $x_2,\dots,x_d=0$ line cannot increase $d_\mathrm{DW}(\,\cdot\,,\bx)$.
\end{proof}

The following proposition gives decay estimates for $u^{(D)}_r$ and $u^{(D)}_\ell$.

\begin{proposition}[\textbf{Decay estimates for the Dirichlet modes}]\mbox{}  \label{prop:decay_dirichlet}\\
	For every $\varepsilon\ge0$ there exist $C_\varepsilon>0$ and $c_\varepsilon>0$ such that
	\begin{equation} \label{eq:agmon_norm_estimate}
	\Big\| e^{(1-\varepsilon)d_\mathrm{DW}(\,\cdot\,,\bx)}u_r^{(D)}\Big\|_{H^1}=\Big\| e^{(1-\varepsilon)d_\mathrm{DW}(\,\cdot\,,.-\bx)}u_\ell^{(D)}\Big\|_{H^1}\le C_\varepsilon
	\end{equation}
	and
	\begin{equation} \label{eq:agmon_norm_estimate_gradient}
	\Big\| e^{(1-\varepsilon)d_\mathrm{DW}(\,\cdot\,,\bx)}\nabla u_r^{(D)}\Big\|_{L^2}=\Big\| e^{(1-\varepsilon)d_\mathrm{DW}(\,\cdot\,,-\bx)}\nabla u_\ell^{(D)}\Big\|_{L^2}\le C_\varepsilon,
	\end{equation}
	where $u^{(D)}_r$, respectively $u^{(D)}_\ell$, is the ground state of \eqref{eq:dirichlet_r}, respectively \eqref{eq:dirichlet_ell}, (extended to zero outside of its domain of definition) for $c=c_\varepsilon$.
\end{proposition}

The importance of this result is the following: even though in a region at distance of order 1 from $-\bx$ the total potential $V_\mathrm{DW}+\lambda w*|u_+|^2$ is of order 1, nonetheless $u^{(D)}_r$ is as small as the exponential of minus the Agmon distance from $\bx$. Compared to the estimates proved previously, it confirms that $u^{(D)}_r$ does not see the left well at all. This is the key to prove~\eqref{eq:L1_convergence} with such a good rate. 

We need the following well-known lemma, which vindicates the importance of~\eqref{eq:second_property_d}: 

\begin{lemma}[\textbf{Computing with the Agmon distance}] \label{lemma:F+-}\mbox{}\\
	Let $\Omega\subset\mathbb{R}^d$ be open with regular boundary, $v\in C^0\big(\overline\Omega,\mathbb{R}\big)$, $\Phi:\overline{\Omega}\to\mathbb{R}$ locally Lipschitz and $u\in C^2\big(\overline\Omega,\mathbb{R}\big)$ with $u_{|\,\partial \Omega}=0$ (including $\lim_{|x|\to\infty}u(x)=0$ if $\Omega$ is unbounded). Let $\nabla\Phi$ be defined in $L^\infty$ as the limit of a mollified sequence $\nabla\Phi_\varepsilon$. Define also 
	$$h:=-\Delta+v.$$
	Then
	\begin{equation*}
	\int_\Omega\big|\nabla \big(e^\Phi u\big)\big|^2\,dx+\int_\Omega\big(v-|\nabla\Phi|^2\big)e^{2\Phi}|u|^2\,dx=\int_\Omega e^{2\Phi} u \,(hu)\,dx.
	\end{equation*}
	Moreover, assume $v-|\nabla\Phi|^2=F_+^2-F_-^2$ with $F_+,F_-\ge0$, and define $F:=F_++F_-$. Then
	\begin{equation} \label{eq:estimate_F+-}
	\int_\Omega\big|\nabla\big(e^\Phi u\big) \big|^2\,dx+\frac{1}{2}\int_\Omega\big|F_+e^\Phi u|^2\,dx\le \int_\Omega\big| F^{-1}e^{\Phi} hu\big|^2\,dx+\frac{3}{2}\int_\Omega \big| F_- e^\Phi u \big|^2\,dx.
	\end{equation}
\end{lemma}

	This is similar to Lemma~\ref{lemma:first_agmon_lemma}. See~\cite[Theorem~1.5]{Agmon-82},~\cite[Lemma 2.3]{HelSjo-85} or \cite[Theorem 3.1.1]{Helffer-88}. We are now ready to prove Proposition \ref{prop:decay_dirichlet}.

\begin{proof}[Proof of Proposition \ref{prop:decay_dirichlet}]
	We will estimate the norms containing $u_r^{(D)}$ only, since, by reflection symmetry, the identities in \eqref{eq:agmon_norm_estimate} and \eqref{eq:agmon_norm_estimate_gradient} are trivial. We will apply \eqref{eq:estimate_F+-} with the following choices (recall that $c$ is the constant that appears in \eqref{eq:dirichlet_r}, its choice will be specified later on):
	\begin{equation*}
	\begin{split}
	\Omega=\;&\Big\{ x\;|\;x_1\ge -\frac{L}{2}+c \Big\}\\
	v=\;&V_\mathrm{DW}+\lambda w*|u_+|^2-\mu^{(D)}\\
	\Phi=\;&(1-\varepsilon)d_\mathrm{DW}(\,\cdot\,,\bx)\\
	u=\;&u^{(D)}_r.
	\end{split}
	\end{equation*}
	We now explain how to choose the functions $F_+,F_-$ and the constant $c$. The main idea is to define $F_+^2$ to be equal to the function $v-|\nabla\Phi|^2$ on the set where $v-|\nabla\Phi|^2$ is larger than some fixed arbitrary positive constant $\kappa$, and to be identically equal to the same $\kappa$ on the set where $v-|\nabla\Phi|^2\le \kappa$. To this end, notice first that
	\begin{equation*}
	\begin{split}
	v(x)-|\nabla\Phi(x)|^2=\;&V_\mathrm{DW}(x)+\lambda w*|u_+|^2(x)-\mu^{(D)}-(1-\varepsilon)^2|\nabla d_\mathrm{DW}(x,\bx)|^2\\
	\ge\;& (2\varepsilon-\varepsilon^2) V_\mathrm{DW}(x)-\mu^{(D)},
	\end{split}	
	\end{equation*}
	having used \eqref{eq:second_property_d} and $w\ge0$ in the second step. This shows that $v-|\nabla\Phi|^2$ can be smaller than a fixed constant only in the regions close to $\bx$ and $-\bx$. As a consequence, for every $\varepsilon>0$, there exists $c_\varepsilon>0$ such that $v(x)-|\nabla\Phi(x)|^2\ge \kappa$ for $-L/2+c_\varepsilon\le x_1\le0$. We pick $c$ equal to such a $c_\varepsilon$ in the definition \eqref{eq:dirichlet_r} of $u_r^{(D)}$.
	The only other region in which $v-|\nabla\Phi|^2$ can be small is the region of small $|x-\bx|$. We take care of this by defining
	\begin{equation*}
	F_+^2(x):=\begin{cases}
	v(x)-|\nabla\Phi(x)|^2,\qquad &|x-\bx|\ge\Big( \frac{\kappa+\mu^{(D)}}{2\varepsilon-\varepsilon^2} \Big)^{1/s}\quad\text{and}\quad x_1\ge -\frac{L}{2}+c_\varepsilon\\
	\kappa,\qquad&|x-\bx|\le\Big( \frac{\kappa+\mu^{(D)}}{2\varepsilon-\varepsilon^2} \Big)^{1/s}.
	\end{cases}
	\end{equation*}
	Correspondingly, we define
	\begin{equation*}
	F_-^2(x)=F_+^2(x)-v(x)+|\nabla\Phi(x)|^2\begin{cases}
	=0,\quad &|x-\bx|\ge\Big( \frac{\kappa+\mu^{(D)}}{2\varepsilon-\varepsilon^2} \Big)^{1/s}\quad\text{and}\quad x_1\ge-\frac{L}{2}+c_\varepsilon\\
	\le \kappa,\quad&|x-\bx|\le\Big( \frac{\kappa+\mu^{(D)}}{2\varepsilon-\varepsilon^2} \Big)^{1/s}.
	\end{cases}
	\end{equation*}
	It is then straightforward to verify that, by construction,
	\begin{equation*}
	F_+^2 \ge \kappa,\qquad F_-^2\le \kappa,\qquad \mathrm{supp}(F_-)\subset\Big\{x\;|\; |x-\bx|\le\Big( \frac{\kappa+\mu^{(D)}}{2\varepsilon-\varepsilon^2} \Big)^{1/s}\Big\}.
	\end{equation*}
	We are then ready to apply \eqref{eq:estimate_F+-}, which yields (recall that $hu=0$)
	\begin{equation*}
	\begin{split}
	\int_{\{x_1\ge-L/2+c_\varepsilon\}}\big| &\nabla\big( e^{(1-\varepsilon)d_\mathrm{DW}(\,\cdot\,,\bx)} u^{(D)}_r \big) \big|^2\,dx+\frac{\kappa}{2}\int_{\{x_1\ge-L/2+c_\varepsilon\}}\big| e^{(1-\varepsilon)d_\mathrm{DW}(\,\cdot\,,\bx)}  u^{(D)}\big|^2\,dx\\
	\le\;& \frac{3\kappa}{2}\int_{\big\{|x-\bx|\le\Big( \frac{\kappa+\mu^{(D)}}{2\varepsilon-\varepsilon^2} \Big)^{1/s}\big\}}\big| e^{(1-\varepsilon)d_\mathrm{DW}(\,\cdot\,,\bx)}u_r^{(D)}\big|^2\,dx.
	\end{split}
	\end{equation*}
	As proven in Lemma \ref{lemma:distance}, the function $d_\mathrm{DW}(x,\bx)$ for $x_1\ge0$ depends on the radial coordinate $|x-\bx|$ only. As a consequence, the integral in the right hand side does not depend on $L$, and therefore it is estimated by a ($\varepsilon$-dependent) constant. This completes the proof of \eqref{eq:agmon_norm_estimate}.
	
	In order to prove \eqref{eq:agmon_norm_estimate_gradient} we write
	\begin{equation*}
	\begin{split}
	\Big\| e^{(1-\varepsilon)d_\mathrm{DW}(\,\cdot\,,\bx)}\nabla u_r^{(D)}\Big\|_{L^2}^2\le\;& \Big\| e^{(1-\varepsilon)d_\mathrm{DW}(\,\cdot\,,\bx)} u_r^{(D)}\Big\|_{H^1}^2+\Big\|\big(\nabla d_\mathrm{DW}(\,\cdot\,,\bx)\big) e^{(1-\varepsilon)d_\mathrm{DW}(\,\cdot\,,\bx)} u_r^{(D)}\Big\|_{L^2}^2\\
	\le\;&\Big\|\sqrt{V_\mathrm{DW}}e^{(1-\varepsilon)d_\mathrm{DW}(\,\cdot\,,\bx)} u_r^{(D)}\Big\|_{L^2}^2+C_\varepsilon,
	\end{split}
	\end{equation*}
	where the second inequality follows from \eqref{eq:second_property_d} and \eqref{eq:agmon_norm_estimate}. Using Lemma \ref{lemma:distance}, $d_\mathrm{DM}(\,\cdot\,,\bx)$ has at least polynomial growth, and therefore one deduces that, for every $\delta>0$, there exists $K_\delta>0$ such that
	\begin{equation*}
	\sqrt{V_\mathrm{DW}}\le K_\delta\, e^{\delta\, d_\mathrm{DW}(\,\cdot\,,\bx)}.
	\end{equation*}
	We then deduce,
	\begin{equation*}
	\begin{split}
	\Big\| e^{(1-\varepsilon)d_\mathrm{DW}(\,\cdot\,,\bx)}\nabla u_r^{(D)}\Big\|_{L^2}^2\le\;&K_\delta^2\Big\|e^{(1-\varepsilon+\delta)d_\mathrm{DW}(\,\cdot\,,\bx)} u_r^{(D)}\Big\|_{L^2}^2+C_\varepsilon\\
	\le\;&K_\delta^2 C_{\varepsilon-\delta}+C_\varepsilon,
	\end{split}
	\end{equation*}
	which proves \eqref{eq:agmon_norm_estimate_gradient} if we fix $\delta<\varepsilon$.
\end{proof}

\subsection{Quasi-modes construction and proof of  $L^1$ and $L^2$ estimates  }

Now we use linear combinations of $u^{(D)}_r$ and $u^{(D)}_\ell$ as quasi-modes for the mean-field Hamiltonian $h_\mathrm{DW}$. A proper smoothing (around respectively $x_1=-L/2+c_\varepsilon$ and $x_1=L/2-c_\varepsilon$) is required. To this end, define a smooth localization function $\chi_r$ such that
\begin{equation*}
\begin{split}
\chi_r(x)=\;&0\qquad x_1\le-\frac{L}{2}+ 2c_\varepsilon
\\
\chi_r(x)=\;&1\qquad x_1\ge-\frac{L}{2}+ 3c_\varepsilon\\
0\le\chi_r(x)\le\;&1,
\end{split}
\end{equation*}
and the corresponding $\chi_\ell(x)=\chi_r(-x_1,x_2,\dots,x_d)$. Define then
\begin{equation}
\psi_r:=\chi_r u^{(D)}_r\qquad \psi_\ell:=\chi_\ell u^{(D)}_\ell
\end{equation}
and
\begin{equation}
r_r:=\big(h_\mathrm{DW}-\mu^{(D)}\big)\psi_r\qquad r_\ell:=\big(h_\mathrm{DW}-\mu^{(D)}\big)\psi_\ell.
\end{equation}
A direct calculation gives
\begin{equation}  \label{eq:expansionrr}
r_r=-2\nabla\chi_r\cdot \nabla u_r^{(D)}-(\Delta\chi_r)u_r^{(D)},
\end{equation}
and therefore
\begin{equation*}
\mathrm{supp}\,(r_r)\subset \{x\;|\;2c_\varepsilon\le|x+\bx|\le 3c_\varepsilon\}.
\end{equation*}
This means that $\psi_r$ and $\psi_\ell$ are quasi-modes for $h_\mathrm{DW}$, the only error coming from the region where, respectively, $\chi_r$ and $\chi_\ell$ differ from zero and one. 

\begin{lemma} [\textbf{Estimates for quasi-modes}]\mbox{} \label{lemma:estimates_for_quasimodes}\\
	We have, with $T$ the tunneling parameter~\eqref{eq:tunnel}
	\begin{align}
	\|r_r\|_{L^2}=\|r_\ell\|_{L^2}\le\;& C_\varepsilon T^{1-\varepsilon} \label{eq:norm_r}\\
	\big|\langle\psi_r,\psi_r\rangle-1\big|=\big|\langle\psi_\ell,\psi_\ell\rangle-1\big|\le\;& C_\varepsilon T^{2-\varepsilon}  \label{eq:norm_psi}\\
	|\langle\psi_r,r_r\rangle|=|\langle\psi_\ell,r_\ell\rangle|\le\;&C_\varepsilon T^{2-\varepsilon} \label{eq:product_psi_r}\\
	0\le\langle\psi_r,\psi_\ell\rangle\le\;& C_\varepsilon T^{1-\varepsilon}. \label{eq:estimate_psi_r_psi_l}
	\end{align}
\end{lemma}

\begin{proof}
	As usual, we can consider only the right functions. To prove \eqref{eq:norm_r}, let us start from \eqref{eq:expansionrr}. Multiplying and dividing by $e^{(1-\varepsilon)d_\mathrm{DW}(\,\cdot\,,\bx)}$ we find
	\begin{equation*} 
	\begin{split}
	\|r_r\|_{L^2}\le\;&  \big\|e^{-(1-\varepsilon)d_\mathrm{DW}(\,\cdot\,,\bx)} (\Delta \chi_r)e^{(1-\varepsilon)d_\mathrm{DW}(\,\cdot\,,\bx)}u^{(D)}_r\big\|_{L^2}\\
	&+ 2 \big\|e^{-(1-\varepsilon)d_\mathrm{DW}(\,\cdot\,,\bx)} (\nabla \chi_r)e^{(1-\varepsilon)d_\mathrm{DW}(\,\cdot\,,\bx)}\nabla u^{(D)}_r\big\|_{L^2}\\
	\le\;&C\Big( \big\|e^{(1-\varepsilon)d_\mathrm{DW}(\,\cdot\,,\bx)}u_r^{(D)}\big\|_{L^2}+\big\|e^{(1-\varepsilon)d_\mathrm{DW}(\,\cdot\,,\bx)}\nabla u_r^{(D)}\big\|_{L^2}  \Big)\\
	&\qquad\times\sup_{2c_\varepsilon\le x_1+L/2\le 3c_\varepsilon}e^{-(1-\varepsilon)d_{\mathrm{DW}}(x,\bx)}.
	\end{split}
	\end{equation*}
	The two norms inside the parenthesis were estimated in Proposition \ref{prop:decay_dirichlet}. To estimate the supremum, we deduce from \eqref{eq:distance_second_estimate} that
	\begin{equation*}
	\begin{split}
	\sup_{2c_\varepsilon\le x_1+L/2\le 3c_\varepsilon}e^{-(1-\varepsilon)d_{\mathrm{DW}}(x,\bx)}\le\;& \sup_{2c_\varepsilon\le x_1+L/2\le 3c_\varepsilon} e^{-(1-\varepsilon)\big( 2A\big(\frac{L}{2}\big)-A\big( \big| \frac{L}{2}+x_1 \big| \big) \big)}\\
	\le\;& C_\varepsilon e^{-2(1-\varepsilon)A\big( \frac{L}{2} \big)},
	\end{split}
	\end{equation*}
	and this proves \eqref{eq:norm_r}.
	
	To prove \eqref{eq:norm_psi} let us notice that
	\begin{equation}
	\begin{split}
	\big| \langle\psi_r,\psi_r\rangle-1 \big|=\;&\int_{\mathbb{R}^d}( 1-\chi_r^2)|u^{(D)}_r|^2\,dx\\
	\le\;& \int_{2c_\varepsilon\le x_1+L/2\le 3c_\varepsilon} |u^{(D)}_r|^2\,dx.
	\end{split}
	\end{equation}
	We then argue as above by multiplying and dividing by $e^{(1-\varepsilon)d_\mathrm{DW}(\,\cdot\,,\bx)}$. The same can be done to prove \eqref{eq:product_psi_r}.
	
	Finally, let us prove \eqref{eq:estimate_psi_r_psi_l} (notice that the positivity of the scalar product is trivial). We write
	\begin{equation*}
	\begin{split}
	\langle\psi_r,\psi_\ell\rangle =\;& \int \chi_r\chi_\ell u^{(D)}_r u^{(D)}_\ell\,dx\\
	\le\;&\sup_{-L/2+2c_\varepsilon\le x_1\le L/2-2c_\varepsilon}\Big(e^{-(1-\varepsilon)d_\mathrm{DW}(x,\bx)-(1-\varepsilon)d_\mathrm{DW}(x,-\bx)}\Big)\\
	&\qquad\times\int \chi_r\chi_\ell\, e^{(1-\varepsilon)d_\mathrm{DW}(\,\cdot\,,\bx)}u^{(D)}_r e^{(1-\varepsilon)d_\mathrm{DW}(\,\cdot\,,-\bx)} u^{(D)}_\ell\,dx.
	\end{split}
	\end{equation*}
	Using Cauchy-Schwartz and then \eqref{eq:agmon_norm_estimate} we see that the integral in the right hand side is estimated by an $\varepsilon$-dependent constant. To estimate the supremum we write
	\begin{equation*}
	\begin{split}
	\sup_{-L/2+2c_\varepsilon\le x_1\le L/2-2c_\varepsilon}&\Big(e^{-(1-\varepsilon)d_\mathrm{DW}(x,\bx)-(1-\varepsilon)d_\mathrm{DW}(x,-\bx)}\Big)\\
	\le\;& \sup_{-L/2+2c_\varepsilon\le x_1\le 0}\Big(e^{-(1-\varepsilon)d_\mathrm{DW}(x,\bx)-(1-\varepsilon)d_\mathrm{DW}(x,-\bx)}\Big)\\
	&\qquad+\sup_{0\le x_1\le L/2-2c_\varepsilon}\Big(e^{-(1-\varepsilon)d_\mathrm{DW}(x,\bx)-(1-\varepsilon)d_\mathrm{DW}(x,-\bx)}\Big)\\
	\le\;& \sup_{-L/2+2c_\varepsilon\le x_1\le 0}e^{-(1-\varepsilon)\big( 2A\big(\frac{L}{2}\big)-A(|L/2+x_1|) +A(|x+\bx|)\,\big)}\\
	&\qquad+\sup_{0\le x_1\le L/2-2c_\varepsilon} e^{-(1-\varepsilon)\big( 2A\big(\frac{L}{2}\big)-A(|L/2-x_1|) +A(|x-\bx|)\,\big)},
	\end{split}
	\end{equation*}
	where the last inequality follows from \eqref{eq:distance_first_estimate} and \eqref{eq:distance_second_estimate}. However, since the function $A$ is monotone increasing, we have
	\begin{equation*}
	A(|L/2+x_1|) \le A(|x+\bx|) \qquad\text{and}\qquad A(|L/2-x_1|) \le A(|x-\bx|),
	\end{equation*}
	and therefore we find
	\begin{equation*}
	\sup_{-L/2+2c_\varepsilon\le x_1\le L/2-2c_\varepsilon}\Big(e^{-(1-\varepsilon)d_\mathrm{DW}(x,\bx)-(1-\varepsilon)d_\mathrm{DW}(x,-\bx)}\Big)\le e^{-2(1-\varepsilon)A\big(\frac{L}{2}\big)},
	\end{equation*}
	which completes the proof.
\end{proof}

Let us now define the orthogonal projections
\begin{equation*}
P_{\pm}:=\ket{u_+}\bra{u_+}+\ket{u_-}\bra{u_-}\qquad\text{and}\qquad P^\perp_\pm=\1-P_\pm.
\end{equation*}
Our aim is an estimate for the norm of $P_{\pm}^\perp \psi_r$ and $P_\pm^\perp \psi_\ell$. Let us start with the following 

\begin{lemma}[\textbf{Further bounds on $\mu_+$}]\label{lemma:dirichlet_energy}\mbox{}\\
	We have
	\begin{equation*}
	|\mu_+- \mu^{(D)}|\le C_\varepsilon T^{1-\varepsilon}.
	\end{equation*}	
\end{lemma}

\begin{proof}
	An upper bound is deduced by taking $\psi_r$ as trial function for the $\mu_+$-minimization problem:
	\begin{equation*}
	\mu_+\le \frac{1}{\langle\psi_r,\psi_r\rangle}\langle \psi_r,h_\mathrm{DW}\psi_r\rangle = \mu^{(D)}+\frac{\langle\psi_r,r_r\rangle}{\langle\psi_r,\psi_r\rangle}\le \mu^{(D)}+C_\varepsilon T^{2-\varepsilon},
	\end{equation*}
	where the second inequality follows from \eqref{eq:product_psi_r} and \eqref{eq:norm_psi}. A suitable lower bound, in turn, was already proven in \eqref{eq:lower_bound_u_+}, i.e.,
	\begin{equation*}
	\mu_+\ge 2\langle\chi_{x_1\ge0} u_+, h_\mathrm{DW} \chi_{x_1\ge0} u_+\rangle-C_\varepsilon T^{1-\varepsilon}\ge \mu^{(D)}-C_\varepsilon T^{1-\varepsilon},
	\end{equation*}
	where the second inequality follows once again by the variational principle for the Dirichlet minimization.
\end{proof}

As a consequence of Lemma \ref{lemma:dirichlet_energy}, and of our main results on the gaps \eqref{eq:first_gap} and \eqref{eq:second_gap}, we see that $\mu^{(D)}$ is asymptotically close to $\mu_+$ (and therefore to $\mu_-$), and hence it is separated from the rest of the spectrum of $h_\mathrm{DW}$ by a gap of order one. We can then write
\begin{equation*}
P^\perp_\pm\psi_r=-\frac{1}{2\pi i}\oint_{\Gamma} \Big( \frac{1}{\mu^{(D)}-z}-\frac{1}{h_\mathrm{DW}-z} \Big)\,dz\,\psi_r,
\end{equation*}
where $\Gamma$ is a closed contour in the complex plane that encircles $\mu_+$, $\mu_-$, and $\mu^{(D)}$, staying at a finite distance both from them and from the rest of the spectrum. A simple calculation yields
\begin{equation}
P^\perp_\pm\psi_r=-\frac{1}{2\pi i}\oint_{\Gamma} \frac{1}{(\mu^{(D)}-z)(h_\mathrm{DW}-z)} \,dz\,r_r.
\end{equation}
By our choice of the contour we have
\begin{equation*}
|\mu^{(D)}-z|^{-1}\le C,\quad\text{and}\quad \big\|\big( h_\mathrm{DW}-z\big)^{-1}\big\|_\mathrm{op} \le C
\end{equation*}
uniformly for $z\in\Gamma$. Hence, recalling \eqref{eq:norm_r}, we find
\begin{equation}\label{eq:projected}
\begin{split}
\|P^\perp_\pm\psi_r\|_{L^2}\le\;& C_\varepsilon T^{1-\varepsilon}\\
\|P^\perp_\pm\psi_\ell\|_{L^2}\le\;& C_\varepsilon T^{1-\varepsilon}.
\end{split}
\end{equation}
Now, define
\begin{equation}\label{eq:psi +-}
\psi_+:= \frac{\psi_r+\psi_\ell}{\|\psi_r+\psi_\ell\|_{L^2}} \qquad\text{and}\qquad\psi_-:= \frac{\psi_r-\psi_\ell}{\|\psi_r-\psi_\ell\|_{L^2}}.
\end{equation}
We have
\begin{equation*}
\big|\|\psi_r+\psi_\ell\|_{L^2}^2-2\big|\le \big|\|\psi_r\|_{L^2}^2+\|\psi_\ell\|_{L^2}^2-2+2\langle\psi_r,\psi_\ell\rangle\big|\le C_\varepsilon T^{1-\varepsilon},
\end{equation*}
where the last inequality follows from \eqref{eq:norm_psi} and \eqref{eq:estimate_psi_r_psi_l}. Similarly,
\begin{equation*}
\big|\|\psi_r-\psi_\ell\|_{L^2}^2-2\big|\le C_\varepsilon T^{1-\varepsilon}.
\end{equation*}
Hence the norms in the denominators of~\eqref{eq:psi +-} satisfy
\begin{align*}
\left\|\psi_r+\psi_\ell\right\|_{L^2} &= \sqrt{2} + O(T^{1-\eps})\\
\left\|\psi_r-\psi_\ell\right\|_{L^2} &= \sqrt{2} + O(T^{1-\eps})
\end{align*}
and combining with~\eqref{eq:projected} we deduce 
$$ \psi_+ = a u_+ + b u_- + O_{L^2}(T^{1-\eps})$$ 
for complex numbers $a,b$. But, since $\psi_+,u_+$ are even under reflections across $x_1 = 0$ and $u_-$ is odd, this must reduce to 
$$ \psi_+ = u_+ + O_{L^2}(T^{1-\eps}).$$
Similarly 
$$ \psi_- = u_- + O_{L^2}(T^{1-\eps}).$$
These are our vindications of~\eqref{eq:annonce +}-\eqref{eq:annonce -}, as in~\cite{HelSjo-85}.
We deduce from the above that  
$$ |u_+| ^2 - |u_-|^2 = 2 \psi_\ell \psi_r + O_{L^1} (T^{1-\eps}) $$
and~\eqref{eq:L1_convergence} then follows from~\eqref{eq:estimate_psi_r_psi_l}. 

To deduce \eqref{eq:L2_convergence} let us first recall that, if $x_1\ge0$, then $u_-(x)$ is positive by Lemma \ref{lemma:positivity_u-}, and $u_+(x)$ is positive by general arguments. This allows to write
\begin{equation*}
\begin{split}
\int_{\mathbb{R}^d}\big| |u_+|-|u_-| \big|^2\,dx=\;&2\int_{x_1 \ge 0} \big| u_+-u_- \big|^2\,dx \\
\le\;&6\int_{x_1 \ge 0} \big| u_+-\psi_+ \big|^2\,dx +6\int_{x_1 \ge 0} \big| \psi_+-\psi_- \big|^2 \,dx\\
&+6\int_{x_1 \ge 0} \big| \psi_--u_- \big|^2\,dx.
\end{split}
\end{equation*}
The estimates for the first and third summand follow already from what we discussed above. For the second summand we write
\begin{equation*}
\begin{split}
\int_{x_1 \ge 0} \big| \psi_+-\psi_- \big|^2\,dx\le\;& \Big[ \|\psi_r+\psi_\ell\|_{L^2}^{-1}-\|\psi_r-\psi_\ell\|_{L^2}^{-1} \Big]^2 \int_{x_1 \ge 0} |\psi_r|^2\,dx\\
&+\Big[ \|\psi_r+\psi_\ell\|_{L^2}^{-1}+\|\psi_r-\psi_\ell\|_{L^2}^{-1} \Big]^2 \int_{x_1 \ge 0}|\psi_\ell|^2\,dx.
\end{split}
\end{equation*}
The first square bracket in the right hand side is smaller than $C_\varepsilon T^{1-\varepsilon}$ by the estimates above. For the integral of $\psi_\ell$ we write
\begin{equation*}
\begin{split}
\int_{x_1\ge0}|\psi_\ell|^2\,dx \le\;& \int_{x_1 \ge 0}|u_\ell^{(D)}|^2\,dx\\
\le\;& \sup_{0	\le x_1\le L/2-c_\varepsilon} e^{-2(1-\varepsilon)d_\mathrm{DW}(x,-\bx)}\big\| e^{(1-\varepsilon)d_\mathrm{DW}(\,\cdot\,,-\bx)}u_\ell^{(D)}\big\|_{L^2}^2\le C_\varepsilon T^{1-\varepsilon},
\end{split}
\end{equation*}
where the last inequality follows from Lemma \ref{lemma:distance} and Proposition \ref{prop:decay_dirichlet}. This proves \eqref{eq:L2_convergence}.

\subsection{Proof of the $L^\infty$ estimate}

In order to prove the $L^\infty$ proximity in \eqref{eq:Linfty_convergence} we will improve the $L^2$ result \eqref{eq:L2_convergence} to an estimate for the $H^2$ norms. Notice that, in the notations of Proposition \ref{prop:first_convergence}, the $L^2$ convergence \eqref{eq:L2_convergence} implies
\begin{equation}  \label{eq:optimal_localized_L2}
\begin{split}
\big\|\chi_{x_1\ge0}\big(u_+-u_-\big)\big\|_{L^2}^2\le\;& C_\varepsilon T^{1-\varepsilon}\\
\big\|\chi_{x_1\le0}\big( u_++ u_-\big)\big\|_{L^2}^2\le\;& C_\varepsilon  T^{1-\varepsilon},
\end{split}
\end{equation}
which also means that \eqref{eq:l_r_L2_norm_u-} holds with an improved rate. We will improve this result to a higher Sobolev norm.

\begin{proposition}[\textbf{$H^2$ convergence}]\mbox{} \label{prop:H2}
	\begin{align} \label{eq:optimal_H2_r}
	\big\|\chi_{x_1\ge0}\big(u_+-u_-\big)\big\|_{H^2}^2\le\;& C_\varepsilon  T^{1-\varepsilon}\\
	\big\|\chi_{x_1\le0} \big(u_++ u_-\big)\big\|_{H^2}^2\le\;& C_\varepsilon  T^{1-\varepsilon}. \label{eq:optimal_H2_ell}
	\end{align}
\end{proposition}

The $L^\infty$ estimate \eqref{eq:Linfty_convergence} is an immediate consequence of Proposition \ref{prop:H2} thanks to the Sobolev embedding
\begin{equation*}
\|f\|_{L^\infty(\mathbb{R}^d)}\le C \|f\|_{H^2(\mathbb{R}^d)},
\end{equation*}
that holds for $d=1,2,3$.
In order to prove Proposition \ref{prop:H2} we start with a Lemma.

\begin{lemma}[\textbf{Commuting $h_{\mathrm{DW}}$ and $\chi_{x_1\ge0}$}]\mbox{} \label{lemma:commut}
	\begin{align*}
	\big\|h_{\mathrm{DW}}\chi_{x_1\ge0}(u_+-u_-)\big\|_{L^2}^2 \le\big\|\chi_{x_1\ge0} h_\mathrm{DW}(u_+-u_-)\big\|_{L^2}^2+C_\varepsilon T^{1-\varepsilon}.
	\end{align*}
\end{lemma}

\begin{proof}
	We have
	\begin{equation*}
	\begin{split}
	\big\langle \chi_{x_1\ge0}(u_+-u_-),h^2_{\mathrm{DW}}\chi_{x_1\ge0}(u_+-u_-)\big\rangle=\;&\big\langle\chi_{x_1\ge0} h_\mathrm{DW}(u_+-u_-),\chi_{x_1\ge0} h_\mathrm{DW}(u_+-u_-)\big\rangle\\
	&+\big\langle\big[ h_\mathrm{DW},\chi_{x_1\ge0}\big](u_+-u_-),\big[ h_\mathrm{DW},\chi_{x_1\ge0}\big](u_+-u_-)\big\rangle\\
	&+2\big\langle\big[ h_\mathrm{DW},\chi_{x_1\ge0}\big](u_+-u_-), h_\mathrm{DW}\chi_{x_1\ge0}(u_+-u_-)\big\rangle\\
	=\;&\big\langle\chi_{x_1\ge0} h_\mathrm{DW}(u_+-u_-),\chi_{x_1\ge0} h_\mathrm{DW}(u_+-u_-)\big\rangle\\
	&+3\big\langle\big[ h_\mathrm{DW},\chi_{x_1\ge0}\big](u_+-u_-),\big[ h_\mathrm{DW},\chi_{x_1\ge0}\big](u_+-u_-)\big\rangle\\
	&+2\big\langle\big[ h_\mathrm{DW},\chi_{x_1\ge0}\big](u_+-u_-), \chi_{x_1\ge0}(\mu_+u_+-\mu_-u_-)\big\rangle.
	\end{split}
	\end{equation*}
	We then have to estimate
	\begin{equation*}
	\begin{split}
	\mathrm{Err}_1=\;&3\big\langle\big[ h_\mathrm{DW},\chi_{x_1\ge0}\big](u_+-u_-),\big[ h_\mathrm{DW},\chi_{x_1\ge0}\big](u_+-u_-)\big\rangle\\
	\mathrm{Err}_2=\;&2\big\langle\big[ h_\mathrm{DW},\chi_{x_1\ge0}\big](u_+-u_-), \chi_{x_1\ge0}(\mu_+u_+-\mu_-u_-)\big\rangle.
	\end{split}
	\end{equation*}
	Since
	\begin{equation*}
	\big[h_\mathrm{DW},\chi_{x_1\ge0}\big]=\Delta\chi_{x_1\ge0}+2(\nabla\chi_{x_1\ge0})\cdot\nabla,
	\end{equation*}
	we deduce, using \eqref{eq:localized_without_anything} and \eqref{eq:localized_gradients},
	\begin{equation*}
	\begin{split}
	|\mathrm{Err}_1|\le\;& \int\big|(\Delta\chi_{x_1\ge0}+2(\nabla\chi_{x_1\ge0})\cdot \nabla)(u_+-u_-)\big|^2\\
	\le\;&2\int|\Delta\chi_{x_1\ge0}|^2\,|u_+-u_-|^2\\
	&+4\int|\nabla\chi_{x_1\ge0}|^2\,|\nabla(u_+-u_-)|^2\\
	\le\;&C_\varepsilon T^{1-\varepsilon}.
	\end{split}
	\end{equation*}
	The estimate of $\mathrm{Err}_2$ is similar, and this completes the proof.
\end{proof}

\begin{proof}[Proof of Proposition \ref{prop:H2}]
	
	Using Lemma \ref{lemma:delta2} and then Lemma \ref{lemma:commut} we have
	\begin{equation} \label{eq:first_H2_estimate}
	\begin{split}
	\big\|\Delta\chi_{x_1\ge0}(u_+-u_-)\big\|_{L^2}^2\le\;&
	\big \|h_\mathrm{DW}\chi_{x_1\ge0}(u_+-u_-)\big\|_{L^2}^2+C\big\|\chi_{x_1\ge0}(u_+-u_-)\big\|_{L^2}^2\\
	\le\;&\big\| \chi_{x_1\ge0}h_\mathrm{DW}(u_+-u_-)\big\|_{L^2}^2 +C\big\|\chi_{x_1\ge0}(u_+-u_-)\big\|_{L^2}^2+C_\varepsilon T^{1-\varepsilon}\\
	=\;&\big\| \chi_{x_1\ge0}(\mu_+u_+-\mu_-u_-)\big\|_{L^2}^2 +C\big\|\chi_{x_1\ge0}(u_+-u_-)\big\|_{L^2}^2+C_\varepsilon  T^{1-\varepsilon}.
	\end{split}
	\end{equation}
	The norm $\|\chi_{x_1\ge0}(u_+-u_-)\|_{L^2}^2$ was already estimated in \eqref{eq:optimal_localized_L2}. To estimate the first term in the right hand side, we expand
	\begin{equation*}
	\big\| \chi_{x_1\ge0}(\mu_+u_+-\mu_-u_-)\big\|_{L^2}^2=\frac{\mu_++\mu_-}{2}-2\mu_+\mu_-\langle \chi_{x_1\ge0} u_+,\chi_{x_1\ge0} u_-\rangle.
	\end{equation*}
	Since \eqref{eq:optimal_localized_L2} implies
	\begin{equation*}
	-2\langle\chi_{x_1\ge0} u_+,\chi_{x_1\ge0} u_-\rangle \le -1+C_\varepsilon T^{1-\varepsilon},
	\end{equation*}
	we deduce
	\begin{equation*}
	\begin{split}
	\big\| \chi_{x_1\ge0}(\mu_+u_+-\mu_-u_-)\big\|_{L^2}^2 \le\;& \frac{\mu_++\mu_-}{2}-\mu_+\mu_-+C_\varepsilon  T^{1-\varepsilon}\\
	=\;&\frac{1}{2}(\mu_--\mu_+)^2+C_\varepsilon T^{1-\varepsilon}\\
	\le \;&C_\varepsilon  T^{1-\varepsilon},
	\end{split}
	\end{equation*}
	where the last step follows from the upper bound in \eqref{eq:first_gap}. This proves \eqref{eq:optimal_H2_r}. A reflection across the $\{x_1=0\}$ hyperplane sends $\chi_{x_1\ge0}(u_+-u_-)$ into $\chi_{x_1\le0}(u_++u_-)$ and thus \eqref{eq:optimal_H2_ell} also follows.
\end{proof}

\appendix

\section{Tunneling terms}\label{sect:app}

Here we deduce a variety of useful bounds from the decay estimates of Section~\ref{sec:decay}:

\begin{proposition}[\textbf{Bounds on tunneling terms}]\mbox{}\\  \label{prop:localized_bounds}
	Let $R\ge0$ be a fixed number. For any $\varepsilon>0$ there exist $c_\varepsilon, C_\varepsilon$ such that, for $L$ large enough,
	\begin{align} \label{eq:scalar_product}
	\int_{\mathbb{R}^d} u_\ell u_r\,dx\le\;& C_\varepsilon T^{1-\varepsilon} \\
	\int_{-R \le x_1\le R } |u_\pm(x)|^2\,dx\le\;&C_\varepsilon T^{1-\varepsilon}  \label{eq:localized_without_anything}\\
	\int_{-R \le x_1\le R } |u_\mathrm{ex}(x)|^2\,dx\le\;&C_\varepsilon T^{1-\varepsilon}  \label{eq:localized_without_anything_ex}\\
	\label{eq:nabla_nabla_term}
	\int_{\mathbb{R}^d} \nabla u_\ell \cdot\nabla u_r\,dx \le\;& C_\varepsilon T^{1-\varepsilon}\\
	\int_{-R \le x_1\le R }|\nabla u_\pm(x)|^2\,dx\le\;&C_\varepsilon T^{1-\varepsilon}  \label{eq:localized_gradients}\\
	\int_{-R \le x_1\le R }|\nabla u_\mathrm{ex}(x)|^2\,dx \le\;&C_\varepsilon T^{1-\varepsilon} \label{eq:localized_gradients_ex}\\
	\int_{\mathbb{R}^d} V_\mathrm{DW}u_\ell u_r\,dx \le\;&C_\varepsilon T^{1-\varepsilon} \label{eq:V_l_r_term} \\
	\label{eq:lr_on_rl}
	\int_{x_1\ge -R}|u_\ell(x)|^2dx=\int_{x_1\le R}|u_r(x)|^2\,dx\le\;& C_\varepsilon T^{1-\varepsilon}\\
	\begin{split}
	\int_{-R\le x_1\le R}|u_\pm(x)|^2(V_r(x)-V_\mathrm{DW}(x))\,dx=\;&\int_{-R\le x_1\le R}|u_\pm|^2(V_\ell(x)-V_\mathrm{DW}(x))\,dx\\\le\;& C_\varepsilon T^{1-\varepsilon}\end{split}  \label{eq:V_DW_to_V}\\
	\int_{-R\le x_1\le R}|u_\mathrm{ex}(x)|^2(V_r(x)-V_\mathrm{DW}(x))\,dx\le\;& C_\varepsilon T^{1-\varepsilon}  \label{eq:V_DW_to_V_r_ex}\\
	\int_{-R\le x_1\le R}V_\mathrm{DW}(x)|u_+(x)|^2\,dx\le\;&C_\varepsilon T^{1-\varepsilon}   \label{eq:chichi_V_+}\\
	\int_{x_1\ge -R} V_r|u_\ell|^2\,dx=\int_{x_1\le R} V_\ell|u_r|^2\,dx\le\;&C_\varepsilon T^{1-\varepsilon} \label{eq:lr_on_rl_V}\\
	\int_{-R\le x_1\le R}|u_+(x)|^2\big(w*|u_+|^2(x)\big)\,dx\le\;&C_\varepsilon T^{1-\varepsilon}   \label{eq:chichi_w_+}\\
	\label{eq:convol_two_sides}
	\iint_{\mathbb{R}^d\times\mathbb{R}^d}w(x-y)|u_\ell(x)|^2|u_r(y)|^2\,dxdy\le\;& C_\varepsilon T^{1-\varepsilon}.
	\end{align}
	
\end{proposition}

\begin{proof}
	\eqref{eq:scalar_product} was already proven in \cite[Proposition (3.3)]{RouSpe-16}. The main point is to use the upper bounds in \eqref{eq:decay_u_r} and \eqref{eq:decay_u_ell} in order to reduce the aim to estimating an integral of the form
	\begin{equation} \label{eq:I_a}
	I_a=\int_{\mathbb{R}^d}e^{-a|x-\bx|^{1+s/2}-a|x-\bx|^{1+s/2}}dx
	\end{equation}
	with $a=(1+s/2)^{-1}-\varepsilon$, the $\varepsilon$ being used to absorb any polynomial correction due to $V$. As said, the estimate of $I_a$ can then be found in \cite[Proposition (3.3)]{RouSpe-16}. The integrals in \eqref{eq:localized_without_anything} and \eqref{eq:localized_without_anything_ex} can in the same way be bounded by an integral of the type $I_a$.
	
	To prove \eqref{eq:nabla_nabla_term} we write
	\begin{equation*}
	\begin{split}
	\int_{\mathbb{R}^d} \nabla u_\ell \nabla u_r\,dx=-\int_{\mathbb{R}^d}u_\ell\Delta u_r\,dx=\int_{\mathbb{R}^d}u_\ell\big(\mu_r-V_r-\lambda w*|u_r|^2\big)u_r\,dx.
	\end{split}
	\end{equation*}
	We can then reduce ourselves to an integral of the form \eqref{eq:I_a} by slightly changing the value of $\varepsilon$ in order to absorb the corrections coming from $V_r$ and $\lambda w*|u_r|^2$. The same holds for every other term from \eqref{eq:localized_gradients} to \eqref{eq:chichi_w_+}.
	
	To prove \eqref{eq:convol_two_sides} we use the fact that $w$ is bounded with compact support to write
	\begin{equation*} 
	\begin{split}
	\iint_{\mathbb{R}^d\times\mathbb{R}^d}w(x-y)|u_\ell(x)|^2|u_r(y)|^2dxdy \le\;& C\iint_{\{|x-y|\le C\}}|u_\ell(x)|^2 |u_r(y)|^2dxdy\\
	=\;& C\iint_{\{|x-y|\le C\}\cap\{x_1\le 0\}}|u_\ell(x)|^2 |u_r(y)|^2dxdy\\
	&+C\iint_{\{|x-y|\le C\}\cap\{x_1\ge 0\}}|u_\ell(x)|^2 |u_r(y)|^2dxdy.
	\end{split}
	\end{equation*}
	The second summand is bounded by $C\int_{x_1\ge0} |u_\ell|^2$, and hence a bound for it follows from \eqref{eq:lr_on_rl}. For the first summand we write
	\begin{equation*}
	\iint_{\{|x-y|\le C\}\cap\{x_1\le 0\}}|u_\ell(x)|^2 |u_r(y)|^2dxdy\le \iint_{\{|x-y|\le C\}\cap\{x_1\le 0\}\cap\{y_1\le C\}}|u_\ell(x)|^2 |u_r(y)|^2dxdy.
	\end{equation*}
	The right hand side is estimate by $C \int_{x_1\le0} |u_r|^2$, for which we can again use \eqref{eq:lr_on_rl}
\end{proof}

\section{Higher spectrum}\label{app:higher}

Most of the proof of Theorem \ref{thm:higher} consists of suitable adaptations of the arguments already used to derive Theorem \ref{thm:main}. We only discuss the modifications due to the fact that higher eigenvalues of the left and right wells are typically degenerate (see \cite{Helffer-88,DimSjo-99} and references therein for more general situations). We do not pursue convergence rates nor the dependence on $k$ (i.e. we only deal with energies $O(1)$ apart from the ground state when $L\to \infty$).

The case $k=1$ of Theorem \ref{thm:higher} follows from the proof of Theorem \ref{thm:main}). We proceed by induction and sketch the proof that the same statements holds for $k$ provided we know it does for all smaller integers. The necessary bounds on the decay of eigenfunctions are similar to what was discussed previously and shall not be reproduced.
 
Using the induction hypothesis one can construct a trial state 
$$ 
\frac{u^\ell_{k} + u^r_{k} }{\sqrt{2}} + \mbox{ correction },
$$ 
where $u^\ell_{k}$ is an eigenfunction corresponding to the $k$-th eigenvalue of $h^{\ell}$ and $u^r_{k}$ its reflection. The correction term can be made very small, and renders the above orthogonal to all eigenfunctions of $\hDW$ corresponding to eigenvalues below $\mu_{2M(k-1) + 1}$. Evaluating the energy of the trial state using decay estimates as in the main text yields 
$$ \mu_{2M(k-1) + 1} \leq \mu^\ell_k + o (1).$$ 
 Localizing the corresponding eigenfunction in the left and right wells as in the main text one deduces that $ u_{2M(k-1) + 1}$ converges in the left well to $U^\ell_{k}$, an eigenfunction corresponding to the $k$-th eigenvalue of $h^{\ell}$. In view of the symmetry of the double-well Hamiltonian we are at liberty to assume that $ u_{2M(k-1) + 1}$ is even or odd with respect to reflections across the potential barrier. Hence  $ u_{2M(k-1) + 1}$ converges in the right well to $\pm$ the reflection of $U^\ell_{k}$. We assume it converges to $U^\ell_{k}$, with obvious modifications in the sequel if it is the other way around.
 
Next one can construct a trial state of the form  
$$ 
\frac{U^\ell_{k} - U^r_{k} }{\sqrt{2}} + \mbox{ correction },
$$ 
where $U^\ell_{k}$ is the eigenfunction  $h^{\ell}$ that $u_{2M(k-1) + 1}$ converges to in the left well and $U^r_{k}$ is its reflection. As above this yields an upper bound to  $\mu_{2M(k-1) + 2}$ and allows to deduce that a corresponding eigenfunction must converge  in the left well to some $\widetilde{u}^\ell_{k}$, an eigenfunction corresponding to the $k$-th eigenvalue of $h^{\ell}$. Using the symmetry with respect to reflections and what has just been proved for $ u_{2M(k-1) + 1}$ we deduce that either $\widetilde{u}^\ell_{k} \perp U^\ell_{k}$ and $ u_{2M(k-1) + 2}$ converges to $\pm$ its reflection in the right well or  $\widetilde{u}^\ell_{k} = U^\ell_{k}$ and $ u_{2M(k-1) + 2}$ converges to $-$ its reflection in the right well. 

Assume the latter situation occurs, again with appropriate modifications to the sequel if it is the former. One can then bound $\mu_{2M(k-1) + 3}$ from above with a trial state of the form 
$$ 
\frac{v^\ell_{k} + v^r_{k} }{\sqrt{2}} + \mbox{ correction },
$$ 
where $v^\ell_{k}$ is an eigenfunction  of $h^{\ell}$ orthogonal to the $U^\ell_{k}$ just found, and $v^r_{k} $ its reflection. Similarly as above we deduce that  $u_{2M(k-1) + 3}$ must converge to some $V^\ell_{k}$ in the left well and to $\pm$ its reflection in the right well, with $V^\ell_{k} \perp U^{\ell}_k$ an eigenfunction  for the $k$-th eigenvalue of $h^{\ell}$. The process can then be repeated inductively, and the claimed results follow. The induction stops when all possibilities of constructing mutually orthogonal even/odd trial states for the double-well Hamiltonian out of the eigenstates of the single wells have been exhausted, that is after $m(k)$ iterations. One thus obtains Items (i) and (iii) of the Theorem by arguments similar to the main text. Item (ii) then follows as in Section \ref{sec:second gap} by localizing the next eigenfunction of $\hDW$ in the left/right wells. Since said eigenfunction must be orthogonal to all lower eigenfunctions of $\hDW$, we deduce from what has been proved so far that the localizations must to leading order be orthogonal to the subspace of the single wells Hamiltonians corresponding to all eigenvalues up to the $k$-th one. The claimed energy gap follows. 
%

%

\end{document}